\documentclass[11pt]{article}
\usepackage[utf8]{inputenc}
\usepackage{amsmath,amsthm}
\usepackage {latexsym}
\usepackage{amssymb}

\usepackage[all]{xy}

\usepackage{xcolor}
 \usepackage{ulem}
\usepackage[toc,page]{appendix}

\newcommand{\beal}{\begin{align}}
\newcommand{\enal}{\end{align}}
\newcommand{\bealn}{\begin{align*}}
\newcommand{\enaln}{\end{align*}}
\newcommand{\bear}{\begin{eqnarray}}
\newcommand{\eear}{\end{eqnarray}}
\newcommand{\beeq}{\begin{equation}}
\newcommand{\eneq}{\end{equation}}

\newcommand{\eps}{{\varepsilon}}
\newcommand{\R}{{\mathbb R}}

\newcommand{\Z}{{\mathbb Z}}

\newcommand{\calC}{{\mathcal C}}

\newcommand{\cL}{{\mathcal{L}_1}}

\newcommand{\calO}{{\mathcal{O}}}

\renewcommand{\ln}{\log}

\def\bm{\left[ \begin{array}{cc}}
\def\endm{\end{array}\right]}

\def\calC{{\mathcal C}}

\def\eps{\varepsilon}

\def\bm{\left[\begin{matrix} }
\def\endm{\end{matrix}\right]}

\def\R{{\mathbb R}}

\newtheorem{theorem}{Theorem}
\newtheorem{lemma}[theorem]{Lemma}

\newtheorem{cor}[theorem]{Corollary}
\newtheorem{prop}[theorem]{Proposition}

\theoremstyle{remark}
\newtheorem{remark}[theorem]{Remark}

\def\cL{\mathcal L}

\renewcommand{\hat}{\widehat}
\renewcommand{\epsilon}{\eps}
\renewcommand{\tilde}{\widetilde}
\numberwithin{equation}{section}
\numberwithin{theorem}{section}

\begin{document}

\title{Cost for a  controlled linear KdV equation}
\author{Joachim Krieger\thanks{Bâtiment des Mathématiques, EPFL, Station 8, CH-1015 Lausanne, Switzerland.  E-mail: \texttt{joachim.krieger@epfl.ch}.},\;\; Shengquan Xiang\thanks{Bâtiment des Mathématiques, EPFL, Station 8, CH-1015 Lausanne, Switzerland.  E-mail: \texttt{shengquan.xiang@epfl.ch.}}}
\maketitle
\begin{abstract}
The controllability of the linearized KdV equation with right Neumann control  is  studied in the pioneering work of Rosier \cite{rosier97}. However,  the proof is by contradiction arguments and the value of the observability constant remains unknown, though rich mathematical theories are built on this totally unknown constant.   We introduce a constructive  method that gives the quantitative value of this constant.
\end{abstract}

\section{Introduction}

The goal of this paper is to  give a quantitative cost estimate  of the controlled system
\[
u_t+ u_x+ u_{xxx}=0,\, u(t,0) = u(t,L) =0,  u_x(t,L) = a(t).
\]
\begin{theorem}\label{Thm-main}
Let $L>0$. There exist effectively computable $T_0 = T_0(L)>0$ and  $c=c(L)>0$ such that  for any $T\geq T_0$ the solution $u$ of 
\[
u_t = -u_x - u_{xxx},\, u(t,0) = u(t,L) = u_x(t,L) = 0,\,u(0,x) = u_0(x),
\]
satisfies 
\begin{gather}\label{eq:c}
\int_0^T|u_x(t,0)|^2\,\,dt\geq c\|u_0\|_{L^2(0,L)}^2, \forall u_0\in L^2,  \textrm{ if }L\notin \mathcal{N}; \\
\int_0^T|u_x(t,0)|^2\,\,dt\geq c\|u_0\|_{L^2(0,L)}^2, \forall u_0\in H\subset L^2,  \textrm{ if }L\in \mathcal{N}.\label{eq:c:H}
\end{gather}
\end{theorem}
\noindent Here  $\mathcal{N}$ is the so called  critical length set and is given by 
\begin{equation}
\mathcal{N}= \left\{2\pi\sqrt{\frac{k^2+kl+l^2}{3}}; k, l\in \mathbb{N}^*\right\}. \notag
\end{equation}

Actually following Lions' H.U.M. \cite{Lions-Hilbert} (see also \cite{coron}) the optimal estimate of the observability inequality \eqref{eq:c} (or \eqref{eq:c:H})  implies the exact controllability of the  KdV equation with some optimal control $a(t)\in L^2(0, T)$.   Rosier \cite{rosier97}  proved that such linear controlled system is exactly controllable if and only if $L\notin \mathcal{N}$.  

Though not controllable in critical cases, we can decompose $L^2$ by $H\oplus M$, where the subspaces $H$ and $M$ are controllable and  uncontrollable parts respectively.  
Later on it is proved successively in \cite{coron04, cerpa07, cerpa09} that the nonlinear controlled KdV system, $i.e. \;u_t+ u_x+ u_{xxx}+ uu_x=0$, is locally controllable with some $a(t)=u_x(t, L)$ despite $M$, where the  cost can be estimated by the related observability controllability in $H$.   Many further results  are developed concerning controllability, stability and stabilization  on this classical model, and most of them are based on the value of  the observability constant.  For  example, in  \cite{zuazua02} this value is directly used to get exponential (energy) stability on $L^2$ for  non-critical  cases  and exponential stability on $H$ for critical cases;  though the finite dimensional central manifold $M$ makes the linear system not asymptotically stable, it is shown  in \cite{coron15}  that the nonlinear term as well as the exponential decay on $H$ lead to polynomial stability of the system;  more recently, in  \cite{coron-rivas-xiang-kdv-16} exponential stabilization is achieved by quadratic structure on $M$ and of course the  exponential decay on $H$.

In \cite{rosier97} Rosier used a method due to Bardos-Lebeau-Rauch \cite{Bardo-Lebeau-Rauch}, and only provided the existence of such constant, while the value of it remained open.  Thus it is  important and interesting to give an explicit observability estimate.  
Typical and classical ways of solving cost problems are  moment methods \cite{2009-Tucsnak-Weiss-book}, Lebeau-Robbiano strategy type methods \cite{Lebeau-Robbiano-CPDE}, and Carleman estimates \cite{Fursikov-Imanuvilov-book-1997}. The first two consist in investigating the eigenfunctions and decomposing the states by them, see for example \cite{Lissy-2014, Lebeau-sharp-2016}.  However, in our case  the related eigenfunctions do not form a Riesz basis, due to the fact that the operator is neither self-adjoint nor skew-adjoint. In fact they are  not even  complete in $L^2(0, L)$, see \cite{2019-xiang-SICON}, which prevents us from directly applying those methods.  Due to the existence of the critical length set, it does not seem natural to consider Carleman estimates.
\\

In this paper, we introduce a constructive approach  that quantifies the observability constant.  We concentrate on  the proof of \eqref{eq:c} for non-critical cases,  mainly presented in Section \ref{sec-main-proof}. Then we  comment in Section \ref{sec-critical} that almost the same proof leads to inequality \eqref{eq:c:H} for critical cases.   More precisely, inequality \eqref{eq:c}  can be achieved in two steps.  Let us denote by $S(t)$ the corresponding semi-group of the operator $Au:=-u_x - u_{xxx}, u(t,0) = u(t,L) = u_x(t,L) = 0$.  
\begin{prop}\label{prop:1} Let $K_1\geq 1$, $L\notin\mathcal{N}$. There exists $\gamma = \gamma(L, K_1)>0$ effectively computable such that the set $\mathcal{B}_{\gamma}(K_1)$,
\begin{align*}
\mathcal{B}_{\gamma}= \mathcal{B}_{\gamma}(K_1):= \big\{u\in H^3&(0,L; \mathbb{C});\,\|u\|_{L^2} = 1,\,\|u\|_{H^3}\leq K_1,\,\,u(0) = u(L) = u_x(L) = 0,\\&|u_x(0)|<\gamma,\; \inf_{\lambda\in \mathbb{C}}\|\lambda u - u_x - u_{xxx}\|_{L^2}<\gamma\big\}
\end{align*}
is empty.
\end{prop}
\begin{prop}\label{prop:2} There exist $\bar{K}_1(L)$ and $T_0(L)$ such that for any $\gamma>0$ there is $\epsilon = \epsilon(L, \gamma)>0$ effectively computable with the property that, if there are $u\in L^2(0,L)\backslash\{0\},  K_1\geq \bar{K}_1(L)$,  and $T\geq T_0(L)$  satisfying
\begin{equation}\label{flux-cond}
\int_0^T \big|\big(S(t)u\big)_x(t,0)\big|^2\,\,dt<\epsilon \|u\|_{L^2(0,L)}^2, 
\end{equation}
then $\mathcal{B}_{\gamma}(K_1)$ is not empty.
\end{prop}
In conjunction with the preceding propositions, this then implies that we can set $c= \epsilon\big(L, \gamma(L, \bar{K}_1(L)\big)$ in \eqref{eq:c}  for Theorem \ref{Thm-main}.
\begin{remark}
We only prove Proposition \ref{prop:1} and  Proposition \ref{prop:2} for $L\geq 4$, though the same way of the proof also apply to the other cases. In fact, when $L$ is below $\sqrt{3}\pi$ (which is small than the first critical length $2\pi$), an alternative simple proof in \cite{cerpatu} gives an explicit observability constant, which, for the completeness of the paper, is also presented, see Appendix \ref{appL4}. 
\end{remark}

\section{Some properties of $S(t)$}
From now on we always assume that  $L\geq 4$.  The goal of this section is to develop several properties concerning the smoothing effect of $S(t)$.   All the results stated here will be demonstrated, and all the constants will be explicitly characterized, in Appendix \ref{app-est}.

Due to some compatibility issues, we define the following Sobolev spaces $H_{(0)}^k$ satisfying natural compatibility conditions on the boundary,
\begin{align*}
H_{(0)}^0(0, L)&:= L^2(0, L);\\
H_{(0)}^1(0, L)&:= \{f\in H^1, f(0)=f(L)=0 \}; \\
H_{(0)}^2(0, L)&:= \{f\in H^2, f(0)=f(L)=f'(L)=0 \}; \\
H_{(0)}^3(0, L)&:= \{f\in H^3, f(0)=f(L)=f'(L)=0 \};\\
H_{(0)}^4(0, L)&:= \{f\in H^4\cap H_{(0)}^3, (Af)(0)=(Af)(L)=0 \};\\
H_{(0)}^5(0, L)&:= \{f\in H^5\cap H_{(0)}^3, (Af)(0)=(Af)(L)=(Af)_x(L)=0 \}; \\
H_{(0)}^6(0, L)&:= \{f\in H^6\cap H_{(0)}^3, (Af)(0)=(Af)(L)=(Af)_x(L)=0 \},
\end{align*}
with the same norm as $H^k$:
\[ 
\|f\|_{H^k(0, L)}^2:= \int_0^L |f^{(k)}(x)|^2+ |f(x)|^2 \, dx.
\]
\begin{lemma}\label{lem-sob-emb}
There is a constant $E_m^n$ which only depends on $n< m$ such that 
\begin{equation}
\int_0^L |f^{(n)}(x)|^2\,dx\leq  E_m^n\left(\delta^{m-n}\int_0^L |f^{(m)}(t)|^2 \,dt+ \delta^{-n}\int_0^L |f(t)|^2 \,dt\right), \; \forall \delta\in (0, 1].\notag
\end{equation}
\end{lemma}

Now we are ready to prove the following properties concerning  regularities of the flow $S(t)$.  Suppose that $f_0\in L^2$, $f(t, x)=S(t)f_0$, then simple integration by parts yields
\begin{gather}
\int_0^T\int_0^L f_x^2(t, x)\,dx \,dt\leq \frac{T+L}{3}\int_0^L f_0^2(x)\,dx, \label{fxbound}\\
\int_0^L f^2(t, x)\,dx= \int_0^L f_0^2(x)\,dx-\int_0^T f_x^2(t, 0)\,dt\leq \int_0^L f_0^2(x)\,dx.\label{fluxine}
\end{gather}
The preceding two inequalities tell us that starting from some $L^2$ data the solution will stay in the same space, moreover, on \textit{almost} every (time) $t\in[0, T]$ the solution becomes $H^1(0, L)$ thus gains  regularity.      Actually, similar regularity results hold for arbitrary order:
\begin{lemma}\label{thm-flow-con}
Let $k\in \{0, 1, 2, 3, 4, 5, 6\}$. If the initial data $f_0$ belongs to $H_{(0)}^k$, then the flow $S(t)f_0$ stays in $C([0, T]; H_{(0)}^k(0, L))\cap L^2(0, T; H_{(0)}^{k+1}(0, L))$.
Moreover, there exist constants $F_0^k(L)$ and  $F_1^k(L)$ independent of the choice of $f\in H_{(0)}^k$ and $T\in (0, L]$ such that 
\begin{align}
\|S(t)f_0\|_{C([0, T]; H_{(0)}^k(0, L))}&\leq  F_0^k \|f\|_{H_{(0)}^k(0, L)}, \label{F0}\\
\|S(t)f_0\|_{L^2(0, T; H_{(0)}^{k+1}(0, L))}&\leq   F_1^k \|f\|_{H_{(0)}^k(0, L)}.\label{F1}
\end{align}
\end{lemma}
\begin{remark}
The same type of smoothing results also hold for the nonlinear KdV flow (for example \cite{bsz}).  But this phenomenon only appears for initial boundary value problems,  it does not exist for KdV flow on whole space.
\end{remark}
An immediate consequence is the following smoothing effect result.
\begin{lemma}\label{thm-flow-smooth}
Let $k\in \{1, 2, 3, 4, 5, 6\}$. There exists a constant $F_s^k=F_s^k(L)$ only depending on $L$ such that
\begin{gather}
\|S(t)f_0\|_{H_{(0)}^k(0, L)}\leq  \frac{F_s^k}{t^{k/2}}\|f_0\|_{L^2(0, L)}, \forall t\in (0, T], T\leq L, \label{tleqL} \\
\|S(t)f_0\|_{H_{(0)}^k(0, L)}\leq  \frac{F_s^k}{L^{k/2}}\|f_0\|_{L^2(0, L)}, \forall t\in [L, +\infty).
\end{gather}
\end{lemma}

\begin{remark}
The rate $t^{-k/2}$ in Lemma \ref{thm-flow-smooth} is optimal  by assuming Lemma \ref{thm-flow-con}.   Moreover, both of them can be generalized to $k\in \mathbb{N}$, while more (but similar) efforts are required to get explicit values. 
\end{remark}

For any given $K>0$, let $\mathcal{A}=\mathcal{A}_K$ be
\[\mathcal{A}:= \{u\in H^3(0,L),\,u(0)=u(L)=0,\;\|u\|_{H^3(0,L)}\leq K\}. \] 
Then we have the following simple 
\begin{lemma}\label{lem:compact} 
 There exist $B = B(L, K)$ and a set 
\[
\{f_1,f_2,\ldots, f_B\}\subset H^3(0, L)
\]
such that for each $f\in \mathcal{A}$, there is $f_j$ with 
\[
\|f - f_j\|_{L^2(0,L)}<\frac{\sqrt{2}}{2}. 
\]
\end{lemma}
An immediate consequence if the
\begin{cor}\label{cor:boundon steps}Assume that $\{g_1,g_2,\ldots, g_{P}\}\subset \mathcal{A}$ is orthonormal. Then $P\leq B(L, K)$. 
\end{cor}

\section{Proofs of Proposition \ref{prop:1} and Proposition~\ref{prop:2}}\label{sec-main-proof}
This section is devoted to the proof  of the  two main propositions of this paper.
In the following we will work with a parameter $K$ much bigger than 8, which will eventually determine $K_1 = K_1(L, K)= B^{\frac12}(L, K)K$. For ease of notations,  from now on,  let 
\begin{gather}
\textrm{$\|\cdot\|$  refers to the $L^2$-norm, and $\langle\cdot\rangle$ refers to the $L^2$-inner product}; \notag\\
\textrm{$\Pi_{\{y_1,...,y_n\}}^\perp$   refers to $\Pi_{\{\textrm{span}\{y_1,...,y_n\}\}^{\perp}}$} ;  \notag\\
\textrm{$f= \mathcal{O}(a)$ if $|f|\leq |a|$,  $g= \mathcal{O}_{H}(a)$ if $\|g\|_H\leq |a|$,  $etc$}. \notag
\end{gather}

First, we prove Proposition~\ref{prop:1}, following the procedure in Rosier's proof.  Observe that this proposition is ineffectively true, since if it's false, we can find a sequence $\gamma_n\rightarrow 0$ as well as functions $u_n$ with associated $\lambda_n$ as in the definition of $\mathcal{B}_{\gamma}$ with $\|u_n\|_{H^3}\leq K_1$, and so in particular a subsequence will have $\lambda_n\rightarrow\lambda_*$ as well as $u_n, u_{n,x}$ converge point wise, and also  in $H^3$ weak sense, to some $u_*$ which is as in Rosier's result, and hence results in a contradiction. 
Rendering this effective will require `perturbing Rosier's proof'. 
\begin{proof}[Proof of Proposition \ref{prop:1}]
Assume that $L\notin \mathcal{N}$ and let $u\in \mathcal{B}_{\gamma}$ as in the statement of that proposition, thus there exist $\lambda$ and $f(x)$ such that
\begin{equation}\label{lambdauf}
\lambda u(x)+ u'(x)+u'''(x)= f(x), x\in (0, L), 
\end{equation}
\begin{equation}\label{lambdaufcond}
\|u\|_{L^2}=1, \|u\|_{H^3}\leq K_1, u(0) = u(L) = u_x(L) = 0, |u'(0)|<\gamma, \|f\|_{L^2}<\gamma.
\end{equation}

At first we can get some information about $\lambda$ from the above equation  on $u$. In fact, we get from the preceding equation that 
\begin{equation}
|\lambda|\|u\|_{L^2}- \left(1+\sqrt{E^1_3}\right)\|u\|_{H^3}\leq |\lambda|\|u\|_{L^2}- \|Au\|_{L^2}\leq \|\lambda u+u'+u'''\|_{L^2}< \gamma, \notag
\end{equation}
thus $\lambda$ is bounded by
\begin{equation}
|\lambda|< \gamma+ \left(1+\sqrt{E^1_3}\right) K_1< 1+ \left(1+\sqrt{E^1_3}\right) K_1=: K_2.\notag
\end{equation}
Moreover, direct integration by parts from equation \eqref{lambdauf} yields
\begin{equation}
\lambda \langle u, u\rangle= \langle -u'-u'''+f, u\rangle,  \notag
\end{equation}
which, combined with \eqref{lambdaufcond}, leads to
\begin{equation}
\textrm{Re}(\langle u', u\rangle)=0, \; \textrm{Re}(\langle u''', u\rangle)=\frac{1}{2}|u'|^2(0).  \notag
\end{equation}
Therefore, $\lambda$ is close to the imaginary axis,
\begin{equation}
|\textrm{Re}(\lambda)|\leq 2\gamma.    \notag
\end{equation}

Then we derive further information on $u$ from classical complex analysis.  By extending $u$ and $f$ trivially past the endpoints of the interval $[0,L]$, we obtain a function $u\in H^{\frac32-}(\R)$, which satisfies the relation 
\[
\lambda u+ u'+u'''= f+ u''(0)\delta_0+ u'(0)\delta'_0-  u''(L)\delta_L, x\in \mathbb{R}.
\]
Then, the extended function $u$, via the Fourier(--Laplace) transformation, further satisfies
\[
\hat{u}(\xi)\cdot\big(\lambda + (i\xi) + (i\xi)^3\big) = \alpha - \beta e^{-iL\xi} + \delta i\xi + \hat{f}(\xi),
\]
\[ |\delta|<\gamma, \; |\hat{f}(\xi)|<\gamma L^{1/2} e^{L |\textrm{im}\xi| }, \forall \xi \in \mathbb{C},
\]
where $\alpha:= u''(0), \beta:= u''(L), \delta:= u'(0)$.
It is followed from  Paley-Wiener theorem  that $\hat{u}(\xi)$ and $\hat{f}(\xi)$ are holomorphic  functions when extended on complex valued $\xi$, as $u(x)$ and $f(x)$ are compactly supported.   

We conclude that away from the zeroes of the polynomial $\lambda + (i\xi) + (i\xi)^3$, we have the representation 
\[
\hat{u}(\xi) = i\frac{\alpha - \beta e^{-iL\xi} + \delta i\xi + \hat{f}(\xi)}{p - \xi + \xi^3},\,p = i\lambda. 
\]
Then observe that $(\alpha, \beta)\neq (0,0)$ since otherwise we cannot possibly have the normalisation condition $\|u\|_{L^2} = 1$ (or $\|\hat{u}\|_{L^2}= 2\pi$), provided $\gamma$ is small enough. In fact, if the function 
\[
i\frac{\delta i\xi + \hat{f}(\xi)}{p - \xi + \xi^3}\in L^2(\R), 
\]
then the polynomial $p - \xi + \xi^3$ has to divide the numerator, in the sense that the quotient is an entire function as well. But since $|p|= |\lambda|<K_2$, the roots of this polynomial lie in a disc of radius $R= R(K_1): = \big(1+3K_2/2\big)^{1/3}>1$ in the complex plane centered at the origin:
\begin{equation}
|\xi|^3=|\xi-p|< K_2+ |\xi|<K_2+ \frac{1}{3}|\xi|^3+\frac{2}{3}.   \notag
\end{equation} 

Choose $\eta\in D_{2R(K_1)}$ and $\Gamma_{3R} = \partial D_{3R(K_1)}$, we get from Cauchy's integral formula that
\begin{align*}
|\hat{u}(\eta)|&=\Big|i\frac{\delta i\eta + \hat{f}(\eta)}{(p - \eta + \eta^3)}\Big|\\
&=\Big|\frac{1}{2\pi i}\int_{\Gamma_{3R}}i\frac{\delta i\zeta + \hat{f}(\zeta)}{(p - \zeta + \zeta^3)(\zeta - \eta) }\,d\zeta\Big|.
\end{align*}
On the one hand, since
\begin{equation}
|\hat{f}(\zeta)|\leq e^{3LR}\int_0^L |f(x)|\,dx< e^{3LR} L^{1/2} \gamma, \forall \zeta\in D_{3R},   \notag
\end{equation}
\begin{equation}
|(p - \zeta + \zeta^3)(\zeta - \eta)|\geq (26K_2+\frac{52}{3})R= \frac{52}{3} R^4, \forall \zeta\in \partial D_{3R}, \forall \eta\in D_{2R},   \notag
\end{equation}
we have
\begin{align}
|\hat{u}(\eta)|&\leq \frac{1}{2\pi}\Big|\int_{\Gamma_{3R}}\frac{\delta i\zeta}{(p - \zeta + \zeta^3)(\zeta - \eta) }\,d\zeta\Big|+  \frac{1}{2\pi}\Big|\int_{\Gamma_{3R}}\frac{\hat{f}(\zeta)}{(p - \zeta + \zeta^3)(\zeta - \eta) }\,d\zeta\Big|,  \notag\\
&\leq \left(\frac{27}{52 R^2} +\frac{9L^{1/2}e^{3LR}}{52R^3}\right)\gamma. \label{etasmall}
\end{align}
On the other hand,   for $\forall \xi\in \big( D_{2R}\big)^c$ we have
\begin{equation}
\lvert p-\xi-\xi^3\lvert\geq \frac{2}{3}\lvert \xi\lvert^3-\frac{1}{3}\lvert \xi\lvert\geq \frac{16}{3}R^3-\frac{2}{3}R>\frac{14}{3}R^3,    \notag
\end{equation}
\begin{equation}
\frac{\lvert \xi\lvert}{\lvert p-\xi-\xi^3\lvert}\leq \frac{\lvert \xi\lvert}{\frac{2}{3}\lvert \xi\lvert^3-\frac{1}{3}\lvert \xi\lvert}< \frac{12}{7\lvert \xi\lvert^2},    \notag
\end{equation}
thus
\begin{equation}
|\hat{u}(\xi)|\leq  \frac{12\gamma}{7\lvert \xi\lvert^2}+ \frac{3}{14}\big|\hat{f}(\xi)\big|, \;\forall \xi\in \big( D_{2R}\big)^c,   \notag
\end{equation}
which, together with \eqref{etasmall}, yields
\begin{align*}
\int_{\mathbb{R}} |\hat{u}(\xi)|^2 d\xi&= \int_{\xi\in [-2R, 2R]} |\hat{u}(\xi)|^2 d\xi+ \int_{\xi\in [-2R, 2R]^c} |\hat{u}(\xi)|^2 d\xi, \\
&\leq 4R\left(\frac{27}{52 R^2} +\frac{9L^{1/2}e^{3LR}}{52R^3}\right)^2\gamma^2+\frac{24}{49}\gamma^2   +\frac{9}{98}\gamma^2,\\
&\leq \left(4R\left(\frac{27}{52 R^2} +\frac{9L^{1/2}e^{3LR}}{52R^3}\right)^2+ \frac{57}{98}\right)\gamma^2.
\end{align*}
This contradicts our assumption that $\|u\|_{L^2}=1$, provided $\gamma$ small enough:
\begin{equation}
\left(4R\left(\frac{27}{52 R^2} +\frac{9L^{1/2}e^{3LR}}{52R^3}\right)^2+ \frac{57}{98}\right)\gamma^2< 2\pi.    \notag
\end{equation}
 Furthermore, by a simple variation of the preceding argument, we infer the existence of $\alpha_*(L, K_1)>0$ such that 
\[
|\alpha| + |\beta|\geq \alpha_*(L, K_1)
\]
is forced by the normalisation condition on $u$. Indeed,  based on the above estimates with $(\alpha, \beta)=(0, 0)$, when $(\alpha, \beta)\neq(0, 0)$, for $\eta\in D_{2R}$ we have
\begin{align*}
|\hat{u}(\eta)|&=\Big|i\frac{\alpha-\beta e^{-iL\eta}+ \delta i\eta + \hat{f}(\eta)}{(p - \eta + \eta^3)}\Big|\\
&=\Big|\frac{1}{2\pi}\int_{\Gamma_{3R}}\frac{\alpha-\beta e^{-iL\zeta}+ \delta i\zeta + \hat{f}(\zeta)}{(p - \zeta + \zeta^3)(\zeta - \eta) }\,d\zeta\Big|, \\
&\leq  \left(\frac{27}{52 R^2} +\frac{9L^{1/2}e^{3LR}}{52R^3}\right)\gamma+ \frac{9}{52R^3}\Big(|\alpha|+ e^{3LR}|\beta|\Big),
\end{align*}
and for $\xi\in (D_{2R})^c\cap \mathbb{R}$ we have
\begin{equation}
|\hat{u}(\xi)|\leq  \frac{12\gamma}{7\lvert \xi\lvert^2}+ \frac{3}{14}\big|\hat{f}(\xi)\big|+ \frac{12(|\alpha|+|\beta|)}{7\lvert \xi\lvert^3},    \notag
\end{equation}
which imply that
\begin{align*}
\int_{\mathbb{R}} |\hat{u}(\xi)|^2 d\xi&= \int_{\xi\in [-2R, 2R]} |\hat{u}(\xi)|^2 d\xi+ \int_{\xi\in [-2R, 2R]^c} |\hat{u}(\xi)|^2 d\xi, \\
&\leq 8R\left(\frac{27}{52 R^2} +\frac{9L^{1/2}e^{3LR}}{52R^3}\right)^2\gamma^2+ 8R \left(\frac{9}{52R^3}\Big(|\alpha|+ e^{3LR}|\beta|\Big)\right)^2,\\
&\;\; + \frac{36}{49}\gamma^2   +\frac{27}{196}\gamma^2+ \frac{27(|\alpha|+|\beta|)^2}{245},\\
&\leq \left(8R\left(\frac{27}{52 R^2} +\frac{9L^{1/2}e^{3LR}}{52R^3}\right)^2+\frac{171}{196}\right)\gamma^2,\\
&\;\; +\left(\frac{81e^{6LR}}{338R^5}+\frac{27}{245}\right)(|\alpha|+|\beta|)^2.
\end{align*}
Thus 
\begin{gather}
\left(\frac{81e^{6LR}}{169R^5}+\frac{54}{245}\right)(|\alpha|+|\beta|)^2\geq 1,   \notag\\
 \alpha_{*}= \left(\frac{81e^{6LR}}{169R^5}+\frac{54}{245}\right)^{-1/2},  \notag
\end{gather}
provided that
\begin{equation}
 \left(8R\left(\frac{27}{52 R^2} +\frac{9L^{1/2}e^{3LR}}{52R^3}\right)^2+\frac{171}{196}\right)\gamma^2\leq 2\pi-1.    \notag
\end{equation}

Moreover, by shrinking $\gamma_0(L, K_1)$ if necessary this can be easily improved to 
\begin{equation}
|\beta|/2\leq |\alpha|\leq 2|\beta| \textrm{ and }\min\{|\alpha| ,  |\beta|\}\geq \frac{1}{3}\alpha_*(L, K_1). \notag
\end{equation}
Since else we can arrange the numerator not to have any zeroes at all on or near the real axis, while as shown in the beginning that $\lambda= i\R + \calO(2\gamma)$, whence there exists at least one root of the denominator that is near the real axis for small $\gamma$. 
 More precisely,  we only need to show the former relation as it leads to the latter one, if which is not true,   then either
\[ |\alpha|<|\beta|/2 \textrm{ with  }|\beta|\geq 2/3\alpha_{*},  \]
or
\[|\alpha|>2|\beta|\textrm{ with }|\alpha|\geq 2/3\alpha_{*}. \]
For  the first case, the zeroes of the numerator that lie  in $D_R$ satisfy
\begin{equation}
\beta e^{-iL\xi}= \alpha  + \delta i\xi + \hat{f}(\xi),
\end{equation}
thus
\begin{equation}
|\beta|e^{L \textrm{Im}(\xi)}\leq |\beta| |e^{-iL\xi}|\leq \frac{|\beta|}{2}+\gamma R+ \gamma L^{1/2} e^{LR},    \notag
\end{equation}
therefore,
\begin{equation}
\textrm{Im}(\xi)\leq \frac{1}{L}\ln(\frac{3}{4})   \notag
\end{equation}
provided that $\gamma$ satisfies
\begin{equation}
\gamma\left(R+ L^{1/2} e^{LR}\right)\leq \frac{\alpha_*}{6}.    \notag
\end{equation}
While for the latter case, 
\begin{equation}
-\alpha= -\beta e^{-iL\xi}  + \delta i\xi + \hat{f}(\xi),   \notag
\end{equation}
\begin{equation}
 |\alpha|\leq \frac{|\alpha|}{2}e^{L \textrm{Im}(\xi)}+\gamma R+ \gamma L^{1/2} e^{LR},   \notag
\end{equation}
therefore
\begin{equation}
\textrm{Im}(\xi)\geq \frac{1}{L}\ln(\frac{3}{2})   \notag
\end{equation}
provided that $\gamma$  satisfies
\begin{equation}
\gamma\left(R+ L^{1/2} e^{LR}\right)\leq \frac{\alpha_*}{6}.      \notag
\end{equation}
Hence, the zero $\xi$ in $D_R$, which exists as $\hat{u}$ is holomorphic,  should verify $L|\textrm{Im}(\xi)|\geq \ln(\frac{4}{3})$.

On the other hand, we turn to  the zeroes of the denominator, which all lie in $D_R$, 
\[p-\xi+\xi^3=0,\]
where $p\in \R+ i\calO(2\gamma), \xi\in D_R$. Suppose that $p$ is given by $a+ib$ with some  $|a|<K_2$ and  $|b|<2\gamma$, we can find some $\xi_0\in D_R\cap \R$ as  solution of $a-\xi_0+\xi_0^3=0$. Therefore, there exists a solution $\xi=\xi_0+r$ of $a+ib-\xi+\xi^3=0$ with $|r|<3\gamma$, thus $|\textrm{Im}(\xi)|<3\gamma$, which is in contradiction with $L|\textrm{Im}(\xi)|\geq \ln(\frac{4}{3})$ if $\gamma$ verifies $3\gamma L\leq \ln(\frac{4}{3})$.
\\

Consider then the numerator $\alpha - \beta e^{-iL\xi} + \delta i\xi + \hat{f}(\xi)$, we can then assume that all the roots of $\alpha - \beta e^{-iL\xi} + \delta i\xi + \hat{f}(\xi)$ in $D_R$ are simple, and have to be of distance $O_{K_1,L}(\gamma)$ from the roots of 
\[
\alpha - \beta e^{-iL\xi},
\]
which, thanks to the fact that $|\beta|/2\leq |\alpha|\leq 2|\beta|$,  are of the form 
\begin{equation}
\mu_0+ \frac{2\pi n}{L},\textrm{ with $|$Re}(\mu_0)|\leq \frac{\pi}{L}, |\textrm{Im}(\mu_0)|\leq \frac{\ln 2}{L},\textrm{ and }n\in \mathbb{Z}.
\end{equation}

In fact, as
\begin{equation}
\left(\alpha - \beta e^{-iL\xi} + \delta i\xi + \hat{f}(\xi)\right)'=iL\beta e^{-iL\xi}+i \delta-i\hat{(xf)}(\xi)     \notag
\end{equation}
if there exists some double  solution $\xi$ in $D_R$, then
\begin{equation}
\frac{\alpha_*L}{3} e^{-LR}\leq|iL\beta e^{-iL\xi}|\leq \left(1+ \left(\frac{L^3}{3}\right)^{1/2}e^{LR}\right)\gamma,     \notag
\end{equation}
contradiction when
\begin{equation}
\left(1+ \left(\frac{L^3}{3}\right)^{1/2}e^{LR}\right)\gamma< \frac{\alpha_* L}{3} e^{-LR}.     \notag
\end{equation}
Furthermore, if $\mu$ is such a zero of $\alpha - \beta e^{-iL\xi}$ that is in $D_R$,   we pick a circle $\Gamma_r(\mu)$ in the complex plane centred at $\mu$ of radius $r\in(0, \min\{\frac{\pi}{8L}, R\})= (0, \frac{\pi}{8L})$. We prove that under certain conditions, which will be chosen later on,  there is  only  one solution of $\alpha - \beta e^{-iL\xi} + \delta i\xi + \hat{f}(\xi)$ that lies in the domain $\big[-\frac{\pi}{L}+ \textrm{Re}(\mu), \frac{\pi}{L}+ \textrm{Re}(\mu)\big)\times \mathbb{R}$, actually this solution  is inside $\Gamma_r(\mu)$.

At first for any $\xi\in \Gamma_r(\mu)$ we have
\begin{align*}
\big|\alpha - \beta e^{-iL\xi}\big|^2&= \big|(\alpha - \beta e^{-iL\xi})-(\alpha - \beta e^{-iL\mu})\big|^2 \\
&=\big|\beta e^{-iL\mu}\left(e^{-iL\xi_r}-1\right)\big|^2, \\
&= |\alpha|^2  \left((\cos(Lra)e^{Lrb}-1)^2+(\sin(Lra)e^{Lrb})^2\right), 
\end{align*}
where $\xi_r=\xi-\mu= r(a+ib)$ with $a^2+b^2=1$. 
\\
If $|a|\geq\frac{1}{8}$, then $(\sin(Lra)e^{Lrb})^2\geq \left(e^{-Lr}\frac{Lr}{16}\right)^2\geq\left(\frac{Lr}{48}\right)^2$. 
\\
If $|a|\leq \frac{1}{8}$, then $|b|\geq \frac{7}{8}$. If further $b<0$, then $(\cos(Lra)e^{Lrb}-1)^2\geq \left(1-e^{-\frac{7Lr}{8}}\right)^2\geq \left(\frac{Lr}{48}\right)^2$. Else $b>0$, thus $(\cos(Lra)e^{Lrb}-1)^2\geq \left(\frac{1}{2}Lr\right)^2$.\\
Therefore, 
\begin{equation}
\big|\alpha - \beta e^{-iL\xi}\big|\geq \frac{|\alpha| Lr}{48}\geq \frac{\alpha_* Lr}{144}, \forall \xi\in \Gamma_r(\mu),     \notag
\end{equation}
which yields 
\begin{equation}
\big|\alpha - \beta e^{-iL\xi}+ \delta i\xi + \hat{f}(\xi)\big|\geq \frac{\alpha_* Lr}{288}, \forall \xi\in \Gamma_r(\mu),     \notag
\end{equation}
if
\begin{equation}
\gamma\left(R+ L^{1/2}e^{LR}\right)\leq \frac{\alpha_* Lr}{288}.     \notag
\end{equation}

Moreover, under the above condition,  there is no solution in $\big[-\frac{\pi}{L}+ \textrm{Re}(\mu), \frac{\pi}{L}+ \textrm{Re}(\mu)\big)\times \mathbb{R}\setminus D_r(\mu)$. Indeed, for any $\xi\in D_R\cap \big[-\frac{\pi}{L}+ \textrm{Re}(\mu), \frac{\pi}{L}+ \textrm{Re}(\mu)\big)\times \mathbb{R}$, we estimate $\alpha - \beta e^{-iL\xi}$ by two situations: \\
if $|\textrm{Im}(\xi-\mu)|\geq r/2$, then 
\begin{equation}
\big|\alpha - \beta e^{-iL\xi}\big|=|\alpha|\big|e^{-iL(\xi-\mu)}-1\big|\geq \frac{\alpha_*}{3} |e^{L\textrm{Im}(\xi-\mu)}-1|\geq \frac{\alpha_*Lr}{48}; \notag
\end{equation}
if $|\textrm{Re}(\xi-\mu)|\geq r/2$ and $|\textrm{Im}(\xi-\mu)|\leq r/2$, then 
\begin{equation}
\big|\alpha - \beta e^{-iL\xi}\big|=|\alpha|\big|e^{-iL(\xi-\mu)}-1\big|\geq |\alpha|e^{L\textrm{Im}(\xi-\mu)}|  \sin(\textrm{Re}(\xi-\mu)L)\geq \frac{\alpha_*Lr}{24}.\notag
\end{equation}

Next we prove that, shrinking the upper bound on $\gamma$ if necessary, there is exactly one solution inside $\Gamma_r(\mu)$.
As demonstrated before, there is no solution on  $\Gamma_r(\mu)$, therefore the number of solutions (counting multiplicity) inside which is given by 
\begin{equation}
\frac{1}{2\pi i} \int_{\Gamma_r(\mu)} \frac{\left(\alpha - \beta e^{-iL\xi} + \delta i\xi + \hat{f}(\xi)\right)'}{\alpha - \beta e^{-iL\xi} + \delta i\xi + \hat{f}(\xi)} d\xi= \frac{1}{2\pi i} \int_{\Gamma_r(\mu)} \frac{iL\beta e^{-iL\xi}+i \delta-i\hat{(xf)}(\xi)}{\alpha - \beta e^{-iL\xi} + \delta i\xi + \hat{f}(\xi)} d\xi.\notag
\end{equation}
As $\mu$ is the only solution of $\alpha - \beta e^{-iL\xi} =0$ inside $\Gamma_r(\mu)$, we also have 
\begin{equation}
1= \frac{1}{2\pi i} \int_{\Gamma_r(\mu)} \frac{\left(\alpha - \beta e^{-iL\xi}\right)'}{\alpha - \beta e^{-iL\xi}}= \frac{1}{2\pi i} \int_{\Gamma_r(\mu)} \frac{iL\beta e^{-iL\xi}}{\alpha - \beta e^{-iL\xi}} d\xi.\notag
\end{equation}
It suffices to find a sufficient condition such  that 
\begin{align*}
N_r:&=\Big|\frac{1}{2\pi i} \int_{\Gamma_r(\mu)} \frac{iL\beta e^{-iL\xi}+i \delta-i\hat{(xf)}(\xi)}{\alpha - \beta e^{-iL\xi} + \delta i\xi + \hat{f}(\xi)} d\xi-  \frac{1}{2\pi i} \int_{\Gamma_r(\mu)} \frac{iL\beta e^{-iL\xi}}{\alpha - \beta e^{-iL\xi}} d\xi\Big|, \\
&=\Big|\frac{1}{2\pi i} \int_{\Gamma_r(\mu)} -\frac{(iL\beta e^{-iL\xi}) (\delta i\xi + \hat{f}(\xi)) }{(\alpha - \beta e^{-iL\xi} + \delta i\xi + \hat{f}(\xi))(\alpha - \beta e^{-iL\xi})}+ \frac{i\delta-i\hat{(xf)}(\xi)}{\alpha - \beta e^{-iL\xi} + \delta i\xi + \hat{f}(\xi)} d\xi \Big|,
\end{align*}
is strictly smaller than 1, since $N_r$ takes value from integer. 
In fact, under the above conditions, thanks to the above known estimates, we have
\begin{align*}
N_r&\leq \frac{1}{2\pi} \int_{\Gamma_r(\mu)}\Big| \frac{(iL\beta e^{-iL\xi}) (\delta i\xi + \hat{f}(\xi)) }{(\alpha - \beta e^{-iL\xi} + \delta i\xi + \hat{f}(\xi))(\alpha - \beta e^{-iL\xi})}\Big| d\xi+   \frac{1}{2\pi} \int_{\Gamma_r(\mu)}\Big|\frac{i\delta-i\hat{(xf)}(\xi)}{\alpha - \beta e^{-iL\xi} + \delta i\xi + \hat{f}(\xi)} \Big| d\xi,\\
&\leq r\Big(\frac{(L\beta e^{LR})\gamma(R+L^{1/2}e^{LR})}{\left(\frac{\alpha Lr}{48}\right)\left(\frac{\alpha_* Lr}{288}\right)} + \frac{\left(1+ \left(\frac{L^3}{3}\right)^{1/2}e^{LR}\right)\gamma}{\frac{\alpha_* Lr}{288}}\Big),\\
&\leq \frac{\gamma}{r} \cdot\frac{288\cdot48 \beta e^{LR}(R+L^{1/2}e^{LR})}{\alpha_* \alpha L}+ \gamma \cdot\frac{288\left(1+ \left(\frac{L^3}{3}\right)^{1/2}e^{LR}\right)}{\alpha_* L}, \\
&\leq 288\gamma\Big( \frac{96e^{LR}(R+L^{1/2}e^{LR})}{\alpha_* L r}+ \frac{1+ \left(\frac{L^3}{3}\right)^{1/2}e^{LR}}{\alpha_* L} \Big).
\end{align*}
Thus it suffices to let $\gamma$ satisfy
\begin{equation}
288\gamma\Big( \frac{96e^{LR}(R+L^{1/2}e^{LR})}{\alpha_* L r}+ \frac{1+ \left(\frac{L^3}{3}\right)^{1/2}e^{LR}}{\alpha_* L} \Big)<1.   \notag
\end{equation}
\\

We conclude that all the zeroes of the numerator $\alpha - \beta e^{-iL\zeta} + \delta i\zeta+ \hat{f}(\zeta)$ in $D_R$ are of the form 
\[
\mu_0 + k\frac{2\pi}{L} + \calO(r),\,k\in \Z, 
\]
where $\mu_0$ satisfies $|\mu_0|\leq \frac{2\pi}{L}$.   Because  all the zeroes of the denominator, which are  in $D_R$,  should also be solutions of the numerator, amongst those solutions have to be the roots of the polynomial $p - \xi + \xi^3$. Picking $\xi_0$ suitably, we may assume that the roots of this polynomial in $D_R$ are then of the form 
\[
\xi_0,\,\xi_1=\xi_0 + k\frac{2\pi}{L} +  2\calO(r),\,\xi_2=\xi_0 +(k+l)\frac{2\pi}{L} +  2\calO(r)\]
where $k, l$ are positive integers. Observe that necessarily we have 
\[
|\xi_i|\leq R, \;|(k+l)\frac{2\pi}{L}|\leq 2R+1/4.
\]

 Then one infers the system 
\begin{align*}
\xi_0 + \xi_1 + \xi_2 = 0,\,\xi_0\xi_1 + \xi_0\xi_2 + \xi_1\xi_2 = -1,\,\xi_0\xi_1\xi_2 = -p,
\end{align*}
and the first two of these equations yield
\begin{equation}
3\xi_0+\frac{2\pi}{L}(2k+l)+4\calO(r)=0,     \notag
\end{equation}
\begin{equation}
3=(\frac{2\pi}{L})^2(k^2+kl+l^2)+ \frac{2\pi}{L}(28k+10l)\calO(r)+ 36\calO(r^2).     \notag
\end{equation}
Thus
\begin{equation}
\big|L^2-\left(2\pi\sqrt{\frac{k^2 + l^2 + kl}{3}}\right)^2\big|\leq 56 L^2(R+1/8)r+ 36L^2r^2\leq 56L^2(R+1)r.     \notag
\end{equation}
In particular, if $56L^2(R+1)r< \min_{k,l\in \mathbf{N}}\big|L^2-\left(2\pi\sqrt{\frac{k^2 + l^2 + kl}{3}}\right)^2\big|$,  then we have  $\mathcal{B}_{\gamma} = \emptyset$.  Let us remark here that the existence of such $r$ satisfying the preceding condition is guaranteed by the selection of $L$.
\\

In conclusion, we can set  $\gamma=\gamma(L, K_1)$ that satisfies, 
\begin{equation}
K_2= 1+ \left(1+\sqrt{E^1_3}\right) K_1,\; R= \big(1+3K_2/2\big)^{1/3}, \;\alpha_{*}= \left(\frac{81e^{6LR}}{169R^5}+\frac{54}{245}\right)^{-1/2}, \notag
\end{equation}
\begin{equation}
r<\frac{\pi}{8L}, \;56L^2(R+1)r< \min_{k,l\in \mathbf{N}}\big|L^2-\left(2\pi\sqrt{\frac{k^2 + l^2 + kl}{3}}\right)^2\big|, \notag
\end{equation}
\begin{equation}
 \left(8R\left(\frac{27}{52 R^2} +\frac{9L^{1/2}e^{3LR}}{52R^3}\right)^2+\frac{171}{196}\right)\gamma^2\leq 2\pi-1,\notag
\end{equation}
\begin{equation}
\gamma\left(R+ L^{1/2} e^{LR}\right)\leq \frac{\alpha_*}{6},\; \gamma\left(1+ \left(\frac{L^3}{3}\right)^{1/2}e^{LR}\right)< \frac{\alpha_* L}{3} e^{-LR},\notag
\end{equation}
\begin{equation}
\gamma\left(R+ L^{1/2}e^{LR}\right)\leq \frac{\alpha_* Lr}{288}, \;288\gamma\Big( \frac{96e^{LR}(R+L^{1/2}e^{LR})}{\alpha_* L r}+ \frac{1+ \left(\frac{L^3}{3}\right)^{1/2}e^{LR}}{\alpha_* L} \Big)<1.\notag
\end{equation}
\end{proof}
\

Now we turn to the the second proposition and begin with outlining the idea of the proof.  
Assume  that there is a $u,\,\|u\|_{L^2} = 1$ as in Proposition~\ref{prop:2} satisfying flux inequality \eqref{flux-cond}. Heuristically,  we shall now construct a finite dimensional vector space $V\subset H^3(0,L)$ of functions satisfying the desired boundary vanishing conditions and such that 
\begin{equation}\label{eq:approx}
\|\Pi_V A f - Af\|<\gamma\|f\|.
\end{equation}
Moreover,  this vector space admits an orthonormal basis in $\mathcal{A}$, such that $\Pi_{\mathbb{C}V}A|_{\mathbb{C}V}$ has a normalised (complex) eigenfunction with $\|u\|_{H^3}\leq B^{\frac12}(L, K)K =:K_1$, and which then implies Proposition~\ref{prop:2}.  More precisely, thanks to Corollary \ref{cor:boundon steps}, suppose that 
\begin{gather}
\{g_1, g_2,..., g_p\}\textrm{ with }p\leq B(L, K)\textrm{ is an orthonormal basis of } V, \label{cond-1}\\
g_j(0)= g_j(L)= (g_j)_x(L)=0, |(g_j)_x(0)|< \frac{\gamma}{\sqrt{B(L, K)}},\label{cond-2}\\
\|g_j\|=1, \;\|g_j\|_{H^3}\leq K,\label{cond-3}\\
A g_j\in V, \;\forall j\in \{1, 2,..., p-1\},\label{cond-4}\\
\|A g_p- \Pi_V A g_p\|< \gamma,\label{cond-5}
\end{gather}
 then the vector space $V$ and the complex vector space $\mathbb{C}V$ satisfies
 \begin{gather}
 \|\Pi_V Af- Af\|< \gamma \|f\|,\;\forall f\in V, \notag\\
  \|\Pi_{\mathbb{C}V} Af- Af\|< \gamma \|f\|,\;\forall f\in \mathbb{C}V,\notag\\
\Pi_{\mathbb{C}V} A:  \mathbb{C}V\longrightarrow \mathbb{C}V.\notag
 \end{gather} 
As $ \mathbb{C}V$ is of finite dimension, the map $\Pi_{\mathbb{C}V} A$ admits at least one eigenvalue:
\begin{gather}
g:= \sum_{j=1}^p a_j g_j\in \mathbb{C}V, \;\Pi_{\mathbb{C}V}Ag= \lambda g, \;\sum_{j=1}^p |a_j|^2=1,\notag
\end{gather}
which further satisfies, 
\begin{gather}
 \|\lambda g- Ag\|= \|\Pi_{\mathbb{C}V} Ag- A g\|< \gamma\|g\|= \gamma,\notag\\
 \|g\|_{H^3}\leq \sum_{j=1}^p |a_j| \|g_j\|_{H^3}\leq B^{\frac{1}{2}}(L, K)K= K_1,\notag\\
g(0)=g(L)=g_x(L)=0,\; |g_x(0)|\leq \sum_{j=1}^p |a_j| |(g_j)_x(0)|< \gamma,\notag
\end{gather}
therefore $g\in \mathcal{B}_{\gamma}$.  

Keeping in mind  the above essential observation, in the following complete proof we will only need to construct orthonormal functions $\{g_j\}$ verifying conditions \eqref{cond-1}--\eqref{cond-5}.
\\
\begin{proof}[Proof of Proposition \ref{prop:2}]
{\bf{Step 0}}:  In the first part, we present some basic properties of the flow, while some of them are based on the ``smallness'' of the flux. The remaining parts of the proof are basically repeating these key observations. \\

\noindent\textit{Observation} $(i)$.   Set $K_0=K_0(L)$ by $\left(2 F^3_s/K_0\right)^{2/3}=1$,  define $\bar{K}_1(L):= K_1(L, K_0)$, and pick $t_1= t_1(K):= \left(2 F^3_s/K\right)^{2/3}\in (0, 1)$ for $K\geq K_0$.   From now on we will work on $K\geq K_0$, which actually will give us a result slightly stronger than Proposition \ref{prop:2}.  As a consequence,  Proposition \ref{prop:2} will be concluded by selecting $K=K_0$.  Thanks to Lemma \ref{thm-flow-smooth} and Lemma \ref{lem-sob-emb},  the flow $S(t)$  satisfies, for $\forall t\in [t_1, +\infty)$,
\begin{align}
|S(t)f\|_{H^3([0,L])}&\leq \frac{F^3_s}{t_1^{3/2}}\|f\|\leq \frac{K}{2}\|f\|,\notag\\
|S(t)f\|_{H^6([0,L])}&\leq \frac{F^6_s}{t_1^3}\|f\|\leq \frac{K^2 F_s^6}{4 (F^3_s)^2}\|f\|, \notag \\
|AS(t)f\|_{L^2(0,L)}&\leq \frac{K(1+\sqrt{E^1_3})\|f\|}{2}< \tilde{K}\|f\|, \notag \\
|AS(t)f\|_{H^3(0,L)}&\leq  \frac{K^2 F_s^6 (1+ \sqrt{E^4_6})\|f\|}{4 (F^3_s)^2}< \tilde{K}\|f\|, \notag \\
|A^2S(t)f\|_{L^2(0,L)}&\leq \left(\frac{K^2 F_s^6 (1+ 2\sqrt{E^4_6})}{4 (F^3_s)^2}+ \frac{K\sqrt{E^2_3}}{2}\right)\|f\|=: \tilde{K}\|f\|. \notag
\end{align}

\noindent\textit{Observation}  $(ii)$. Another important observation is that the  $L^2$ norm of the flow stays close to its initial value, thanks to \eqref{fluxine}:
\begin{equation}
\|f(t)\|=\|f_0\|+ \frac{\calO\left(|\int_0^T f_x^2(t, 0)\,dt|\right)}{\|f_0\|}, \forall t\in [0, T].     \notag
\end{equation}
For instance, if $\|f_0\|=1$ and the flux $|\int_0^T f_x^2(t, 0)\,dt|<a$, then the energy of the flow stays close to 1: $\|S(t)f\|\in [1-a, 1]$.\\

\noindent\textit{Observation} $(iii)$.  Owning to the strong regularity of $S(t)u$, we are able to  estimate $S(s)u-S(t)u$.  More precisely, for any $\delta\in (0, 1/2)$ and any $\forall t\in [t_1, +\infty)$, direct calculation implies,
\begin{gather}
S(t+\delta)f- S(t)f=\int_t^{t+\delta}S'(s)fds= \delta S'(t)f+ \int_t^{t+\delta}\int_t^s S''(r)\, drds, \notag
\end{gather} 
thus
\begin{equation}\label{Stdelta}
\frac{S(t+\delta)f - S(t)f}{\delta}= AS(t)f + \tilde{K}\mathcal{O}_{L^2}(\delta)\|f\|, \forall  t\in [t_1, +\infty).
\end{equation}
\

\noindent\textit{Observation} $(iv)$. Thanks to  the relation \eqref{Stdelta},  we can  estimate the flux of $Af(t)$.  Assuming $|\int_0^T f_x^2(t, 0)\,dt|<a<1/2$ and $\delta\in (0, t_1)$,  we have that,  for any $t\in [t_1, T-t_1-\delta]$, 
\begin{align*}
&\;\;\;\;\;\int_0^{T-t-t_1} \Big(S(s)\big(AS(t)f\big)\Big)_x^2(0)\,ds,\\
&\leq\int_0^{T-t-\delta} \Big(S(s)\big(AS(t)f\big)\Big)_x^2(0)\,ds, \\
&= \int_0^{T-t-\delta} \left(S(s)\left(\frac{S(t+\delta)f - S(t)f}{\delta}- \tilde{K}\mathcal{O}_{L^2}(\delta)\|f\|\right)\right)_x^2(0)\,ds, \\
&= \int_0^{T-t-\delta} \left(\frac{1}{\delta}S(s)\big(S(t+\delta)f\big) - \frac{1}{\delta}S(s)\big(S(t)f\big)- S(s)\left(\tilde{K}\mathcal{O}_{L^2}(\delta)\|f\|\right)\right)_x^2(0)\,ds, \\
&\leq 3\int_0^{T-t-\delta} \left(\frac{1}{\delta}S(s)\big(S(t+\delta)f\big)\right)_x^2(0) + \left(\frac{1}{\delta}S(s)\big(S(t)f\big)\right)_x^2(0)+ \left(S(s)\left(\tilde{K}\mathcal{O}_{L^2}(\delta)\|f\|\right)\right)_x^2(0)\,ds, \\
&\leq 3\tilde{K}^2\delta^2\|f\|^2+ \frac{6}{\delta^2}\int_{t}^{T}\big(S(t')f\big)_x^2(0)\,dt',\\
&\leq 3\tilde{K}^2\delta^2\|f\|^2+ \frac{6a}{\delta^2}.
\end{align*}

\noindent\textit{Observation} $(v)$.  Observe that $Af(t)$ and $f(t)$ are orthogonal, provided  the null flux, $i.e. \int_0^T f_x^2(t, 0)\,dt=0$,
\begin{align*}
\langle Af(t), f(t)\rangle&= \langle -f_x-f_{xxx}, f\rangle=-\frac{f^2_x(t, 0)}{2},\\
\langle Af(t), f(t)\rangle&= \langle \frac{d}{\,dt}f(t), f(t)\rangle= \frac{1}{2}\frac{d}{\,dt} \|f(t)\|^2=0.
\end{align*}
Thus it is natural to have a perturbed version,  $Af(t)$ and $f(t)$ are ``almost'' orthogonal when the flux is small.  Suppose that  $|\int_0^T f_x^2(t, 0)\,dt|<a$,  then for any $ t\in [t_1, T-\delta]$, any $\delta\in (0, 1/2)$, we have
\begin{align*}
&\;\;\;\;\;\langle S(t)f, AS(t)f\rangle,\\
 &=  \left\langle S(t)f,  \frac{S(t+\delta)f - S(t)f}{\delta}- \tilde{K}\mathcal{O}_{L^2}(\delta)\|f\| \right\rangle,\\
&= \tilde{K}\calO(\delta)\|f\|^2 + \left\langle S(t)f,  \frac{S(t+\delta)f - S(t)f}{\delta}\right\rangle,\\
&= \tilde{K}\calO(\delta)\|f\|^2 +\frac{1}{2\delta}\Big(\big\langle S(t)f- S(t+\delta)f,  S(t+\delta)f - S(t)f\big\rangle+ \big\langle  S(t)f+S(t+\delta)f,  S(t+\delta)f - S(t)f\big\rangle\Big),\\
&=  \tilde{K}\calO(\delta)\|f\|^2+\frac{\calO(\delta)}{2} \big\|AS(t)f + \tilde{K}\mathcal{O}_{L^2}(\delta)\|f\| \big\|^2+ \frac{1}{2\delta} \left(\big\|S(t+\delta)f\big\|^2-\big\|S(t)f\big\|^2\right),\\
&= \tilde{K}\calO(\delta)\|f\|^2 + \calO(\delta) \left(\tilde{K}^2+\tilde{K}^2 \delta^2\right)\|f\| ^2+ \frac{\calO(a)}{2\delta}, \\
&=\calO\left(4\delta\tilde{K}^2\|f\|^2+\frac{a}{2\delta}\right)
\end{align*}
which is small provided that $a\ll \delta\ll 1$.
\\

\noindent\textit{Observation} $(vi)$.  If two   small flux flows are orthogonal at the beginning, then they are ``almost'' orthogonal along the flow.  Indeed, suppose that for some $a<1/2$ we have 
\begin{gather}
\int_0^T \big(S(t)f\big)_x^2(0) \,dt<a,\; \int_0^T \big(S(t)g\big)_x^2(0) \,dt<a, \notag
\end{gather}
then direct integration by parts shows
\begin{align*}
\langle S(t)f, S(t)g\rangle- \langle f, g\rangle&=\int_0^t \frac{d}{\,dt}\langle S(s)f, S(s)g\rangle\,ds,\\
&=\int_0^t  \langle AS(s)f, S(s)g\rangle+ \langle S(s)f, AS(s)g\rangle\,ds, \\
&= -\int_0^t\big(S(s)f\big)_x(0)\big(S(s)g\big)_x(0)ds, \\
&= \calO(a), \;\forall t\in [0, T].
\end{align*}

\noindent\textit{Observation} $(vii)$. Let $V$  be a subspace of $L^2$. Then 
\begin{equation}
\|\Pi_{S(t)V}^\perp S(t)f\|\leq \|\Pi_{V}^\perp f\|.     \notag
\end{equation}
Indeed, there exists a $f_1\in V$ such that 
\begin{equation}
f=f_1+g, \;g= \Pi_{V}^\perp f,     \notag
\end{equation}
in view of the linearity of the flow, we have
\begin{equation}
S(t)f= S(t)f_1+ S(t)g.     \notag
\end{equation}
Because  $S(t)f_1\in S(t)V$,  the projection satisfies,
\begin{equation}
\|\Pi_{S(t)V}^\perp S(t)f\|\leq  \|S(t)g\|\leq \|g\|=\|\Pi_{V}^\perp f\|.      \notag
\end{equation}
\begin{remark}
If the flux small condition is replaced by the null flux condition, then all these observations become even better.  
\end{remark}

{\bf{Step 1}}: 
In particular, we have that $u$ and $g_{11}: = S(t_1)u$ satisfy
\begin{gather}
\|g_{11}\|_{H^3}\leq\frac{K}{2},      \notag\\
 \|g_{11}\|_{L^2([0,L])}=1+ \mathcal{O}(\varepsilon),      \notag\\
\frac{S(t_1+\delta)u - S(t_1)u}{\delta} = Ag_{11} +  \tilde{K}\mathcal{O}_{L^2}(\delta), \forall \delta\in (0, 1/2),\label{eq:almostdiff}
\end{gather}
the second inequality on account of the flux assumption on $u$, and  the third is small
for some very small $\delta_1$, which, however, is much larger than $\epsilon$.

If 
\[
\|Ag_{11}\|<\frac{\gamma}{2}, 
\]
then stop.  We further define  $g_{11}(s): = S(t_1+s)u= S(s)g_{11}, s\in [0, t_1]$, which satisfies 
\begin{gather}
\|g_{11}(s)\|_{H^3}\leq\frac{K}{2}, \;\|g_{11}(s)\|_{L^2([0,L])}=1+ \mathcal{O}(\varepsilon),      \notag\\
\|Ag_{11}(s)\|=\|S(s)Ag_{11}\|\leq \|Ag_{11}\|<\frac{\gamma}{2},      \notag\\
\int_0^{t_1} \left(g_{11}(s)\right)_{x}^2(0)ds= \int_{t_1}^{2t_1} (S(t)u)_x^2(0) \,dt\leq \int_{0}^{T} (S(t)u)_x^2(0) \,dt\leq \varepsilon,     \notag
\end{gather}
hence there exists $s$ such that $|g_{11}(s)_x(0)|\leq (\varepsilon/t_1)^{1/2}$. 
Observe that if we set $y_{11}:= \frac{g_{11}(s)}{\|g_{11}(s)\|}$, $V:= \text{span}\{y_1\}$, then
\begin{gather}
\|y_{11}\| = 1,\,\|y_{11}\|_{H^3}\leq K,     \notag\\
y_{11}(0)=y_{11}(L)=(y_{11})_x(L)=0, |(y_{11})_x(0)|< \frac{\gamma}{\sqrt{B(L, K)}},     \notag\\
\|A y_{11}- \Pi_VA y_{11}\|\leq \|A y_{11}\|< \gamma, \;V= \text{span}\{y_{11}\},     \notag
\end{gather}
if $\varepsilon< 1/18$ satisfies
\begin{equation}
\frac{3}{2}\sqrt{\frac{\varepsilon}{t_1}}< \frac{\gamma}{\sqrt{B(L, K)}}.     \notag
\end{equation}
As a consequence,  conditions \eqref{cond-1}--\eqref{cond-5} hold with $p=1, g_1=y_{11}$, which, as it is shown at the beginning, implies that $\mathcal{B}_{\gamma}\neq \emptyset$.

If on the other hand we have 
\[
\|Ag_{11}\|\geq \frac{\gamma}{2}, 
\] 
then proceed to the next step. 
\\

{\bf{Step 2}}: Now we have $\|Ag_{11}\|\geq \frac{\gamma}{2}$.  For ease of  notations, we define the following $L^2$--normalization operator: 
\begin{equation}
\cL: f\longrightarrow \frac{f}{\|f\|}, \;\forall f\neq 0. \notag
\end{equation}
It be can easily checked that $\cL$ verifies 
\begin{equation}
\cL S(t)= \cL S(t) \cL, \;\cL A= \cL A \cL.\notag
\end{equation}
Set 
\begin{equation}
 y_{21} =\cL  A y_{11}= \cL Ag_{11}.     \notag
\end{equation}

First we recall some properties of $y_{11}$ and $y_{21}$, thanks to the observations in \textit{Step} 0:
\begin{gather}
\|y_{11}\|= \|y_{21}\|= 1,\; \sqrt{1-\varepsilon}\leq\|g_{11}\|\leq 1,\; \|Ag_{11}\|\geq \frac{\gamma}{2},\notag\\
|\langle g_{11}, Ag_{11}\rangle|\leq 4\delta\tilde{K}^2+\frac{\varepsilon}{2\delta},\forall \delta\in (0, 1/2), \notag \\
|\langle y_{11}, y_{21}\rangle|\leq   \left(4\delta\tilde{K}^2+\frac{\varepsilon}{2\delta}\right)\frac{4}{\gamma}, \forall \delta\in (0, 1/2),\label{y11-y21}\\
\int_0^{T-t_1}\big(S(s)y_{11}\big)_x^2(0)\,ds= \frac{1}{\|g_{11}\|^2}\int_{t_1}^{T}\big(S(s) u\big)_x^2(0)\,ds\leq 2\varepsilon.\label{y11flux}
\end{gather}
Despite that the normalized function $y_{21}$ may have a poor $H^3$-bound,  its boundary trace at $x = 0$ is small in the sense that, $\forall \delta\in (0, t_1)$,
\begin{gather}
\int_0^{T-2t_1}\big(S(s)Ag_{11}\big)_x^2(0)\,ds= \int_0^{T-2t_1} \Big(S(s)\big(AS(t)u\big)\Big)_x^2(0)\,ds\leq 3\delta^2\tilde{K}^2+ \frac{6\varepsilon}{\delta^2},\notag \\
\int_0^{T-2t_1}\big(S(s)y_{21}\big)_x^2(0)\,ds\leq \frac{1}{\|Ag_{11}\|^2}\int_0^{T-2t_1}\big(S(s)Ag_{11}\big)_x^2(0)\,ds\leq \left(3\delta^2\tilde{K}^2+ \frac{6\varepsilon}{\delta^2}\right)\frac{4}{\gamma^2},\label{y21flux}
\end{gather}
having taken advantage of observation $(iv)$. For the sake of simplicity, we define  an upper bound for  \eqref{y11-y21}, \eqref{y11flux} and \eqref{y21flux}:
\begin{gather}
\mathcal{C}_1=\mathcal{C}_1(\varepsilon, \delta_1, \gamma, \tilde{K}):= \left(\delta_1^2\tilde{K}^2+ \frac{\varepsilon}{\delta_1^2}\right)\frac{24}{\gamma^2}, \forall \delta_1\in (0, \min\{\frac{1}{2}, t_1\}),      \notag
\end{gather}
which can be sufficiently small for  well chosen $\varepsilon$ and $\delta_1$.  To make it clear, $\calC_1$ is an upper bound for 
\begin{equation}
\int_0^{T-2t_1}\big(S(s)y_{11}\big)_x^2(0)\,ds,\; \int_0^{T-2t_1}\big(S(s)y_{21}\big)_x^2(0)\,ds, \; |\langle y_{11}, y_{21}\rangle|.     \notag
\end{equation}

 From the above inequality, we derive the existence of  $\bar{t}_1\in [t_1, 2t_1]$ such that $|(S(\bar{t}_1)y_{11})_x(0)|, |(S(\bar{t}_1)y_{21})_x(0)|\leq (2\calC_1/t_1)^{1/2}$. Set 
 \begin{equation}
 g_{12}:= S(\bar{t}_1)y_{11},\; g_{22}:= S(\bar{t}_1) y_{21}.     \notag
 \end{equation}
They share similar properties as $y_{11}$ and $y_{21}$, thanks to the observations and the flux condition:
\begin{gather}
\|g_{12}\|^2,\|g_{22}\|^2\in [1- \calC_1,1],      \notag\\
\|g_{12}\|_{H^3}, \|g_{22}\|_{H^3}\leq \frac{K}{2}; \;
|(g_{12})_x(0)|, |(g_{22})_x(0)|\leq \left(\frac{2\calC_1}{t_1}\right)^{1/2}, \notag\\
\int_0^{T-4t_1}\big(S(s)g_{12}\big)_x^2(0)\,ds\leq \calC_1,\;
\int_0^{T-4t_1}\big(S(s)g_{22}\big)_x^2(0)\,ds\leq \calC_1,      \notag\\
|\langle g_{12}, g_{22}\rangle|\leq 2 \calC_1,     \notag\\
Ag_{12}=AS(\bar{t}_1)\cL g_{11}\in \text{span}\{g_{22}\}     \notag
\end{gather}

 Finally we can set 
 \begin{equation}
 y_{12}:= \cL g_{12},\; y_{22}:= \cL \Pi^{\perp}_{y_{12}} g_{22}, \; y^{\perp}_{32}:=\Pi^{\perp}_{y_{12}} A y_{22},\; y_{32}:= \cL  y^{\perp}_{32}.     \notag
 \end{equation}
 Notice that $y_{22}$ is obtained from the Gram-Schmidt procedure, thus 
 \begin{equation}
 Ay_{12}\in \text{span}\{g_{12}, g_{22}\}=\text{span}\{y_{12}, y_{22}\}.     \notag
 \end{equation}
It can be proved that $\{y_{12}, y_{22}, y_{32}\}$ share similar properties as $\{y_{11}, y_{21}\}$:  ``small'' flux. It is easy to get for $y_{12}$:
\begin{equation}
\int_0^{T-4t_1}\big(S(s)y_{12}\big)_x^2(0)\,ds\leq 2\calC_1.     \notag
\end{equation}
As for $y_{22}$, it can be written in the form of a ``prepared'' flow: 
\begin{align*}
y_{22}&= \frac{1}{\|\Pi^{\perp}_{y_{12}}g_{22}\|}\left(g_{22}-\frac{g_{12}}{\|g_{12}\|^2} |\langle g_{12}, g_{22}\rangle|\right),\\
&= \frac{1}{\|\Pi^{\perp}_{y_{12}}g_{22}\|} S(\bar t_1)\left(y_{21}- \frac{|\langle g_{12}, g_{22}\rangle|}{|g_{12}\|^2}y_{11}\right),\\
&=: S(\bar t_1) z_{22},
\end{align*}
for which we can successively get,
\begin{equation}
\|\Pi^{\perp}_{y_{12}}g_{22}\|^2= \Big\|g_{22}-\frac{g_{12}}{\|g_{12}\|^2} |\langle g_{12}, g_{22}\rangle|\Big\|^2\in [1-5\calC_1  ,1],     \notag
\notag\end{equation}
\begin{equation}
\|z_{22}\|\leq \frac{1+ \calC_1}{(1- \calC_1)\sqrt{1-5\calC_1}}\leq 2 ,     \notag
\notag\end{equation}
\begin{equation}
\int_0^{T-2t_1}\big(S(s)z_{22}\big)_x^2(0)\,ds\leq  \frac{8\calC_1}{1-5\calC_1}\leq 16\calC_1,     \notag
\notag\end{equation}
\begin{equation}
\int_0^{T-5t_1}\big(S(s)AS(t_1)z_{22}\big)_x^2(0)\,ds\leq 12\tilde{K}^2\delta^2+ \frac{96\calC_1}{\delta^2}, \forall \delta\in (0, t_1). \notag
\end{equation}
Thanks to the above inequalities on $z_{22}$, we can further get
\begin{equation}
\|y_{22}\|_{H^3}\leq \frac{\|\Pi^{\perp}_{y_{12}}g_{22}\|_{H^3}}{\|\Pi^{\perp}_{y_{12}}g_{22}\|}\leq  \frac{1}{\sqrt{1-5\calC_1}}\cdot\frac{K}{2}(1+ \frac{2\calC_1}{1-\calC_1})\leq K \frac{1+ 4\calC_1}{2\sqrt{1-5\calC_1}}\leq \frac{3K}{4},   
\notag\end{equation}
\begin{equation}
\|Ay_{22}\|_{L^2}\leq  \frac{\tilde{K}}{\sqrt{1-5\calC_1}}\leq 2\tilde{K},
\notag\end{equation}
\begin{equation}
\int_0^{T-4t_1}\big(S(s)y_{22}\big)_x^2(0)\,ds\leq  \frac{8\calC_1}{1-5\calC_1}\leq 16 \calC_1, \notag
\end{equation}
if $\calC_1\leq 1/18$.
 
About $y^{\perp}_{32}$ which is related to $Ay_{22}$, we know from its definition that, 
 \begin{equation}
 y^{\perp}_{32}= Ay_{22}- \Pi_{y_{12}} A y_{22}= Ay_{22}- l y_{12} \textrm{ with }|l|\leq 2\tilde{K},      \notag
 \end{equation}
 thus the inner products are
 \begin{gather}
 \langle  y^{\perp}_{32}, y_{12}\rangle=0,      \notag\\
 \langle  y^{\perp}_{32}, y_{22}\rangle=  \langle  Ay_{22}, y_{22}\rangle=  \langle  AS(\bar t_1)z_{22}, S(\bar t_1)z_{22}\rangle\leq 16\delta\tilde{K}^2+\frac{8\calC_1}{\delta}, \forall \delta\in (0, \frac{1}{2}).     \notag
 \end{gather}
 Moreover, the flux of $y^{\perp}_{32}$ can be estimated by 
 \begin{align*}
 \int_0^{T-5t_1}\big(S(s)y^{\perp}_{32}\big)_x^2(0)\,ds&= \int_0^{T-5t_1}\big(S(s)Ay_{22}- l S(s)y_{12} \big)_x^2(0)\,ds,\\
 &\leq 2\int_0^{T-5t_1}\big(S(s)Ay_{22}\big)_x^2(0)\,ds+ 2l^2 \int_0^{T-5t_1} (S(s)y_{12})_x^2(0) \,ds,\\
 &\leq 24\delta^2\tilde{K}^2+ \frac{192\calC_1}{\delta^2}+ 16 \tilde{K}^2\calC_1, \forall \delta\in (0, t_1).
 \end{align*}
 
 If $\|y^{\perp}_{32}\|<\frac{\gamma}{2}$, then stop, and we verify that $\{y_{12}, y_{22}\}$ satisfy conditions \eqref{cond-1}--\eqref{cond-5}, thus $\mathcal{B}_{\gamma}$  is not empty,  provided that 
 \begin{equation}
 (2\calC_1/t_1)^{1/2}< \gamma/\sqrt{B(L, K)}.    \notag
 \end{equation}
 
 If $\|y^{\perp}_{32}\|\geq\frac{\gamma}{2}$, then we define $\calC_2= \calC_2(\calC_1, \delta_2)=\calC_2(\varepsilon,\delta_1, \delta_2, \gamma, \tilde{K})$ by
 \begin{equation}
 \calC_2:= \left(\frac{2}{\gamma}\right)^2\left( 24\delta_2^2\tilde{K}^2+ \frac{192\calC_1}{\delta_2^2}+ 16 \tilde{K}^2\calC_1\right), \forall \delta_2\in (0, \min\{\frac{1}{2}, t_1\}),     \notag
 \end{equation}
 which is an upper bound for 
 \begin{equation}
 \int_0^{T-5t_1} \big(S(s)y_{i 2}\big)_x^2(0)\,ds\; \forall i\in\{1, 2, 3\}, \textrm{ and } |\langle  y_{32}, y_{22}\rangle|.     \notag
 \end{equation}
 Now we proceed to the next step. 
 
 \begin{equation*}
 \xymatrix{
 u\ar@{~>}[r]^{S(t_1)} & g_{11}\ar@{-->}[d]^{A} \ar@{->>}[r]^{\cL}   & y_{11}\ar@{~>}[r]^{S(\bar{t}_1)}\ar[d]^{\cL A}  &g_{12}\ar[r]^{id}   &h_{12}\ar@{->>}[r]^{\cL}                &y_{12}\ar@{~>}[r]^{S(\bar{t}_2)}  \ar@{}|\perp[d]  &g_{13}\ar@{-->>}[r]&  y_{13}\ar@{}|\perp[d]  \ar@{~>}[r]^{S(\bar{t}_3)} & \\
    &Ag_{11}\ar@{->>}[r]^{\cL}     & y_{21}\ar@{~>}[r]^{S(\bar{t}_1)} &g_{22}\ar[r]^{\Pi^{\perp}_{h_{12}}}   &h_{22}\ar@{->>}[r]^{\cL} \ar@{}|(.6)\circlearrowleft[dr]              &y_{22}\ar[d]^{\cL \Pi^{\perp} A} \ar@{~>}[r]^{S(\bar{t}_2)} \ar@{-->}[dll]^{A\;\;\;} &g_{23}\ar@{-->>}[r] & y_{23}\ar@{}|\perp[d] \ar@{~>}[r]^{S(\bar{t}_3)} &  \\
    &                & z_{22}\ar@{~>}[urrr]^{S(\bar{t}_1)}\ar@{-->>}_{AS(\bar{t}_1)}[r] &Ay_{22}\ar[r]_{\Pi^{\perp}_{y_{12}}}&y_{32}^{\perp}\ar@{->>}[r]_{\cL}  &y_{32}\ar@{~>}[r]_{S(\bar{t}_2)}&g_{33}\ar@{-->>}[r]& y_{33}\ar[d]^{\cL \Pi^{\perp} A}\ar@{~>}[r]^{S(\bar{t}_3)} & \\
&&&&&&&y_{43}    \ar@{~>}[r]_{S(\bar{t}_3)} &
    }
\end{equation*}

{\bf{Step 3}}: Here we assume $\|y_{32}^{\perp}\|\geq \frac{\gamma}{2}$, and set 
\begin{align*}
y_{32}:=  \cL \Pi^{\perp}_{y_{12}} A y_{22}.
\end{align*}
Then proceeding as before, we can obtain boundary trace inequality. 
Observe that the projection onto $y_{12}^{\perp}$ introduces a flux term of size at most $O(\epsilon)$ due to our earlier boundary flux estimation for $y_{12}$. 
\\
Since we have lost regularity for $y_{32}$, we regain this by applying the flow $S(\bar{t}_2)$ again, for some  $\bar{t}_2\in [t_1, 2_1]$, resulting in 
\begin{equation}
g_{13}= S(\bar{t}_2) y_{12}, \;g_{23}= S(\bar{t}_2) y_{22}, \;g_{33}= S(\bar{t}_2) y_{32}.   \notag
\end{equation}
Then we apply the Gram-Schmi\,dt procedure, by first orthogonalizing
\begin{equation}
h_{13}=g_{13},\; h_{23}=\Pi_{h_{13}}^{\perp} g_{23}, \; h_{33}=\Pi_{\{h_{13}, h_{23}\}}^{\perp} g_{33}= S(\bar{t}_2)z_{33},     \notag
\end{equation}
and then normalizing
\begin{equation}
y_{13}= \cL h_{13},\; y_{23}= \cL h_{23},\; y_{33}= \cL h_{33}.     \notag
\end{equation}
As the demonstration in Step 2, we are able to estimate $y_{13}, y_{23}, y_{33}$ and $Ay_{33}$ in terms of some $\calC_3$ which only depends on $\calC_2$ and $\delta_3$.  Then if 
\[ y^{\perp}_{43}:=\Pi^{\perp}_{\{y_{13}, y_{23}\}}Ay_{33},\;
\|y^{\perp}_{43}\|<\frac{\gamma}{2}, 
\]
we stop the process. Else we continue iteratively.   We ignore detailed calculation in this step, as it will be covered by the next step for  general cases.
\\

{\bf{Step 4}}: General induction step. In this step we provide general iteration estimates.  At first we prove the following lemma.
\begin{lemma}\label{lem-cn}
Let $n\geq 2$.  For any $0< (n+1)c_n< \min\{1/18, 1/2n\}$, any $T_n\geq 3t_1$, any orthonormal functions $\{y_{1n},..., y_{nn}\}$, and any normal function $y_{(n+1) n}$ satisfying
\begin{gather}
A y_{i n}\in \textrm{span}\{y_{1n},..., y_{nn}\}, \forall i\in\{1, ..., n-1\},\label{ayiinspan}\\ 
A y_{n n}\in \textrm{span}\{y_{1n},..., y_{(n-1) n}, y_{(n+1) n}\}, \label{ayninspan}\\
\langle y_{i n}, y_{(n+1) n}\rangle =0, \forall i\in\{1, ..., n-1\},\\
|\langle y_{n n}, y_{(n+1) n}\rangle|\leq c_n,\\
\int_0^{T_n} \big(S(s)y_{i n}\big)_x^2(0)\,ds\leq c_n, \forall i\in\{1, ..., n+1\}, \label{ssyin}
\end{gather}
we can find orthonormal functions $\{y_{1 (n+1)},..., y_{(n+1) (n+1)}\}$ such that
\begin{gather}
\|y_{i (n+1)}\|_{H^3}\leq 3K/4, \; |(y_{i (n+1)})_x(0)|\leq \frac{3}{2}\sqrt{\frac{(n+1)c_n}{t_1}},\\
\int_0^{T_n-3t_1} \big(S(s)y_{i (n+1)}\big)_x^2(0)\,ds\leq 2c_n, \forall i\in\{1, ..., n+1\},\\
A y_{i (n+1)}\in \textrm{span}\{y_{1(n+1)},..., y_{(n+1)(n+1)}\}, \forall i\in\{1, ..., n\}.\label{Ayin1}
\end{gather}
Moreover, if we further project $Ay_{(n+1) (n+1)}$  on $\textrm{span}\{y_{1(n+1)},..., y_{n(n+1)}\}$,
\begin{equation}
y_{(n+2)(n+1)}^{\perp}:= \Pi_{y_1,..., y_{n (n+1)}}^{\perp} Ay_{(n+1) (n+1)}, 
\end{equation}
then it satisfies,
\begin{equation}
\int_0^{T_n-3t_1} \big(S(s)y_{(n+2) (n+1)}^{\perp}\big)_x^2(0)\,ds\leq \left(\frac{\gamma}{2}\right)^2 c_{n+1}, \notag
\end{equation}
\begin{equation}
|\langle  y_{(n+1) (n+1)}, y_{(n+2) (n+1)}^{\perp}\rangle| \leq  \left(\frac{\gamma}{2}\right)^2 c_{n+1},\notag
\end{equation}
where $c_{n+1}= c_{n+1}\big(c_n, \delta_{n+1}\big)$ is given by
\begin{equation}
 c_{n+1}:= \left(\frac{2}{\gamma}\right)^2(n+1) \left(6\tilde{K}^2  \delta_{n+1}^2+ \frac{12c_n}{\delta_{n+1}^2} +4\tilde{K}^2c_n\right), \forall \delta_{n+1}\in (0, \min\{1/2, t_1\}).     \notag
\end{equation}
\end{lemma}
\begin{proof}
These functions are directly constructed via the Gram-Schmidt procedure.  

It follows from \eqref{ssyin}  that 
\begin{equation}
\int_{t_1}^{2t_1} \Big(\sum_{i=1}^{n+1} (S(s)y_{i n})_x^2\Big)(0)\,ds\leq (n+1)c_n,     \notag
\end{equation}
hence there exists $\bar{t}_n\in [t_1, 2t_1]$ such that 
\begin{equation}\label{boundsmall}
\sum_{i=1}^{n+1} (S(\bar{t}_n)y_{i n})_x^2(0)\leq \frac{(n+1)c_n}{t_1}.
\end{equation}

For $i\in\{1, ..., n+1\}$, we define 
\begin{equation}
g_{i (n+1)}:= S(\bar{t}_n)y_{i n}, 
\end{equation}
which, thanks to the boundary bound condition  \eqref{boundsmall}, the flux condition \eqref{ssyin}, Observation $(i)$, and Observation $(vi)$,  satisfies
\begin{equation}
1-c_n\leq \|g_{i  (n+1)}\|\leq 1, \; \|g_{i  (n+1)}\|_{H^3}\leq \frac{K}{2}, \notag
\end{equation}
\begin{equation}
g_{i (n+1)}(0)= g_{i (n+1)}(L)=(g_{i (n+1)})_x(L)=0, \; |(g_{i  (n+1)})_x(0)|\leq \sqrt{\frac{(n+1)c_n}{t_1}}, \notag 
\end{equation}
\begin{equation}
\int_0^{T_n- 2t_1} \big(S(s)g_{i (n+1)}\big)_x^2(0)\,ds\leq c_n, \notag 
\end{equation}
\begin{gather}
\langle g_{i (n+1)}, g_{j (n+1)}\rangle\leq c_n, \forall\; (j, i)\neq (n, n+1), j<i, \notag \\\langle g_{n (n+1)},
g_{(n+1) (n+1)}\rangle\leq 2c_n.\notag 
\end{gather}
Then we derive from \eqref{ayiinspan}--\eqref{ayninspan} that 
\begin{align*}
A g_{1 (n+1)},..., A g_{(n-1) (n+1)}&\in \textrm{span}\{S(\bar{t}_n)y_{1n},..., S(\bar{t}_n)y_{nn}\},\\
&=\textrm{span}\{g_{1 (n+1)},..., g_{n (n+1)}\}, \notag
\end{align*}
and that 
\begin{align*}
A g_{n (n+1)}= S(\bar{t}_n) Ay_{n n}&\in \textrm{span}\{S(\bar{t}_n)y_{1n},..., S(\bar{t}_n)y_{(n-1)n}, S(\bar{t}_n)y_{(n+1)n}\},\\
&=\textrm{span}\{g_{1 (n+1)},..., g_{(n-1) (n+1)}, g_{(n+1) (n+1)}\}, 
\end{align*}
thus 
\begin{equation}
A g_{i (n+1)}\in \textrm{span}\{g_{1 (n+1)},..., g_{(n+1) (n+1)}\}, \forall i\in\{1, ..., n\}.\notag
\end{equation}

Next, we orthogonalize $\{g_{i (n+1)}\}$  by  $\{h_{i (n+1)}\}$.  More precisely, we find a upper triangular matrix $A_{n+1}= (a_{i j})_{1\leq i, j\leq n+1}$ with $a_{i i}=1$, such that the elements of 
\begin{equation}
(h_{1 (n+1)}, h_{2 (n+1)},..., h_{(n+1) (n+1)}):= (g_{1 (n+1)}, g_{2 (n+1)},..., g_{(n+1) (n+1)}) A_{n+1} \notag
\end{equation}
are orthogonal.  In such a case, the orthonormal functions $\{y_{i (n+1)}\}$ can be chosen by 
\begin{equation}
y_{i (n+1)}:= \cL h_{i (n+1)}= \frac{h_{i (n+1)}}{\|h_{i (n+1)}\|}.     \notag
\end{equation}

In the remaining part of the proof, we check that $\{y_{i (n+1)}\}$ verify the lemma.
Now we need to fix the value of $a_{i j}$.  From the construction of $h_{i (n+1)}$ we know that 
\begin{equation}
\textrm{span}\{h_{1 (n+1)},..., h_{i (n+1)}\}= \textrm{span}\{g_{1 (n+1)},..., g_{i (n+1)}\}, \forall i\in\{1, ..., n+1\}, \notag
\end{equation}
which implies that 
\begin{equation}\label{Ah-or}
A h_{i (n+1)}, A g_{i (n+1)}\in \textrm{span}\{h_{1 (n+1)},..., h_{(n+1) (n+1)}\}, \forall i\in\{1, ..., n\}.
\end{equation}
Moreover, by the definition of $h_{i (n+1)}$,
\begin{equation}
h_{i (n+1)}\perp g_{j (n+1)}, \forall  1\leq j<i\leq n+1, \notag
\end{equation}
hence
\begin{equation}
0= \langle h_{i (n+1)}, g_{j (n+1)}\rangle= \sum_{k=1}^i a_{k i} \langle g_{k (n+1)}, g_{j (n+1)}\rangle,\notag
\end{equation}
which implies
\begin{equation}
|a_{j i}|\leq c_n\left(\sum_{1\leq k\leq i-1}|a_{k i}|\right)+ |\langle g_{i (n+1}), g_{j (n+1)}\rangle|,\notag
\end{equation}
thus
\begin{equation}
 \sum_{1\leq k\leq i-1}|a_{k i}|\leq \frac{i c_n}{1-(i-1)c_n}\leq \frac{(n+1)c_n}{1-nc_n}\leq 2(n+1)c_n , \forall 1<i\leq n+1.     \notag
\end{equation}
Therefore, 
\begin{gather}
\|h_{i (n+1)}\|= \|\sum_{k=1}^i a_{ki} g_{k (n+1)}\|\in [\sqrt{1-c_n}-2(n+1)c_n,  1+2(n+1)c_n]\subset [1-3(n+1)c_n, 1+2(n+1)c_n], \notag\\
\|h_{i (n+1)}\|^{-1}\in [1-2(n+1)c_n, 1+4(n+1)c_n], \forall 1\leq i\leq n+1.     \notag
\end{gather}
Many informations about the orthonormal basis $\{y_{i (n+1)}\}_{i\in \{1, ..., n+1\}}$ can be obtained from such explicit formulas. At first condition \eqref{Ayin1} is guaranteed by \eqref{Ah-or} and the definition of $y_{i (n+1)}$. Then, 
\begin{align}
\|y_{i (n+1)}\|_{H^3}&\leq \|h_{i (n+1)}\|^{-1} \|h_{i (n+1)}\|_{H^3}, \notag\\
&\leq (1+4(n+1)c_n) \sum_{k=1}^i a_{k i} \|g_{k (n+1)}\|_{H^3},  \notag\\
&\leq \frac{K}{2} (1+4(n+1)c_n) (1+ 2(n+1)c_n), \notag \\
&\leq \frac{3K}{4}.     \notag
\end{align}
Similarly, we have 
\begin{equation}
|(y_{i (n+1)})_{x}(0)|\leq \frac{3}{2}\sqrt{\frac{(n+1)c_n}{t_1}},     \notag
\end{equation}
\begin{align}
\int_0^{T_n-2t_1} \Big(S(s)y_{i (n+1)}\Big)_x^2(0)\,ds&\leq (1+ 4(n+1)c_n)^2 \int_0^{T_n-2t_1} \Big(S(s)h_{i (n+1)}\Big)_x^2(0)\,ds, \notag\\
&\leq (1+ 4(n+1)c_n)^2\int_0^{T_n-2t_1} \Big(\sum_{k=1}^i a_{k i} \big(S(s)g_{k (n+1)}\big)_x\Big)^2(0)\,ds,\notag \\
&\leq  (1+ 4(n+1)c_n)^2\int_0^{T_n-2t_1} \Big(\sum_{k=1}^i \sqrt{|a_{k i}|}\cdot \sqrt{|a_{k i}|}\big(S(s)g_{k (n+1)}\big)_x(0)\Big)^2\,ds,\notag \\
&\leq (1+ 4(n+1)c_n)^2\int_0^{T_n-2t_1} (1+2(n+1)c_n)\Big(\sum_{k=1}^i  |a_{k i}|\big(S(s)g_{k (n+1)}\big)_x^2(0)\Big)\,ds,\notag\\
&\leq (1+ 4(n+1)c_n)^2(1+2(n+1)c_n)^2 c_n, \notag\\
&\leq 2c_n.     \notag
\end{align}
It remains to estimate $y_{(n+2) (n+1)}:= Ay_{(n+1) (n+1)}$. Instead of dealing with $Ay_{(n+1) (n+1)}$ directly, we consider some $z_{(n+1) (n+1)}$ such that $S(\bar{t}_n)z_{(n+1) (n+1)}= y_{(n+1) (n+1)}$, therefore  $Ay_{(n+1) (n+1)}$ can be estimated from observations:$(i)$, $(iv)$ and $(v)$.  In fact, 
\begin{align}
y_{(n+1) (n+1)}&=\frac{1}{\|h_{(n+1) (n+1)}\|} \sum_{k=1}^{n+1} a_{k (n+1)} g_{k (n+1)}, \notag\\
&= \frac{1}{\|h_{(n+1) (n+1)}\|} \sum_{k=1}^{n+1} a_{k (n+1)} S(\bar{t}_n) y_{k n}, \notag\\
&= S(\bar{t}_n)\left(\frac{1}{\|h_{(n+1) (n+1)}\|} \sum_{k=1}^{n+1} a_{k (n+1)}y_{k n}\right), \notag\\
&=: S(\bar{t}_n)z_{(n+1) (n+1)},     \notag
\end{align}
while $z_{(n+1) (n+1)}$ satisfies
\begin{gather}
\|z_{(n+1) (n+1)}\|\leq (1+4(n+1)c_n) (1+2(n+1)c_n)\leq \sqrt{2},     \notag\\
\int_0^{T_n} \big(S(s)z_{(n+1) (n+1)}\big)_x^2(0)\,ds\leq (1+4(n+1)c_n)^2 (1+2(n+1)c_n)^2 c_n\leq 2c_n.     \notag
\end{gather}
Thanks to  Observation $(i)$, we have
\begin{equation}
\|Ay_{(n+1) (n+1)}\|= \|AS(\bar{t}_n)z_{(n+1) (n+1)}\|\leq \sqrt{2}\tilde{K}.     \notag
\end{equation}
Thus, 
\begin{align}
y_{(n+2)(n+1)}^{\perp}&= \Pi_{y_1,..., y_{n (n+1)}}^{\perp} Ay_{(n+1) (n+1)}, \notag\\
&= Ay_{(n+1) (n+1)}+ \sum_{k=1}^{n} b_k y_{k (n+1)}, \notag\\
&= AS(\bar{t}_n)z_{(n+1) (n+1)}+ \sum_{k=1}^{n} b_k y_{k (n+1)},     \notag
\end{align}
with 
\begin{equation}
\|\sum_{k=1}^{n} b_k y_{k (n+1)}\|\leq \|AS(\bar{t}_n)z_{(n+1) (n+1)}\|\leq\sqrt{2}\tilde{K},     \notag
\end{equation}
which, together with the orthonormal property of $y_{i (n+1)}$, implies that 
\begin{equation}
\sum_{k=1}^n b_k^2\leq 2\tilde{K}^2.     \notag
\end{equation}

Finally, we are able to get
\begin{align*}
|\langle  y_{(n+1) (n+1)}, y_{(n+2) (n+1)}^{\perp}\rangle|&=|\langle  y_{(n+1) (n+1)}, AS(\bar{t}_n)z_{(n+1) (n+1)}+ \sum_{k=1}^{n} b_k y_{k (n+1)}\rangle|, \\
&= |\langle S(\bar{t}_n)z_{(n+1) (n+1)}, AS(\bar{t}_n)z_{(n+1) (n+1)}\rangle|, \\
&\leq 4\delta \tilde{K}^2\|z_{(n+1) (n+1)}\|^2+ \frac{c_n}{\delta}, \\
&\leq 8\delta\tilde{K}^2+ \frac{c_n}{\delta},  \forall \delta\in (0, 1/2),
\end{align*}
and 
\begin{align*}
&\;\;\;\;\;\int_0^{T_n-3t_1} \big(S(s)y_{(n+2) (n+1)}^{\perp}\big)_x^2(0)\,ds,\\
&= \int_0^{T_n-3t_1} \left(\Big(S(s)AS(\bar{t}_n)z_{(n+1) (n+1)}\Big)_x(0)+ \sum_{k=1}^{n} b_k\Big(S(s) y_{k (n+1)}\Big)_x(0)\right)^2\,ds, \\
&\leq (n+1) \int_0^{T_n-3t_1}\left(S(s)AS(\bar{t}_n)z_{(n+1) (n+1)}\Big)_x^2(0)+ \sum_{k=1}^{n} b_k^2\Big(S(s) y_{k (n+1)}\Big)_x^2(0) \right)\,ds,\\
&\leq (n+1) \left(6\tilde{K}^2  \delta^2+ \frac{12c_n}{\delta^2} +4\tilde{K}^2c_n\right), \forall \delta\in (0, t_1).
\end{align*}
Notice that $\delta_{n+1}^2<\delta_{n+1}$ for any $\delta_{n+1}<\min\{1/2, t_1\}$,  we get the last two inequalities of  Lemma \ref{lem-cn} and complete  its proof.
\end{proof}

Let us define $y_{(n+2) (n+1)}:= \cL y_{(n+2) (n+1)}^{\perp}$. Then it satisfies
\begin{equation}\label{yin1orn2}
\langle  y_{i (n+1)}, y_{(n+2) (n+1)}\rangle=0, \forall i\in \{1,..., n\},
\end{equation}
\begin{equation}\label{smallyn1n2}
|\langle  y_{(n+1) (n+1)}, y_{(n+2) (n+1)}\rangle|\leq \left(\frac{\gamma}{2}\right)^2 \frac{c_{n+1}}{\|y_{(n+2) (n+1)}^{\perp}\|},
\end{equation}
\begin{equation}\label{smallyn2x}
\int_0^{T_n-3t_1} \big(S(s)y_{(n+2) (n+1)}\big)_x^2(0)\,ds\leq \left(\frac{\gamma}{2}\right)^2 \frac{c_{n+1}}{\|y_{(n+2) (n+1)}^{\perp}\|^2},
\end{equation}
where $0<\delta\leq \min\{1/2, t_1\}$.
Suppose that  
\begin{equation}
\frac{3}{2}\sqrt{\frac{(n+1)c_n}{t_1}}< \frac{\gamma}{\sqrt{B(L, K)}}.
\end{equation}

If $\|y_{(n+2) (n+1)}^{\perp}\|<\frac{\gamma}{2}$, then the orthonormal basis $\{y_{i (n+1)}\}_{1\leq i\leq n+1}$ satisfies conditions \eqref{cond-1}--\eqref{cond-5}, thus $\mathcal{B}_{\gamma}\neq \emptyset$.

If $\|y_{(n+2) (n+1)}^{\perp}\|\geq \frac{\gamma}{2}$, then from Lemma \ref{lem-cn} and inequalities \eqref{yin1orn2}--\eqref{smallyn2x} we derive that the orthonormal functions $\{y_{i (n+1)}\}_{1\leq i\leq n+}$ and normal function $y_{(n+2) (n+1)}$ satisfy
\begin{gather}
A y_{i (n+1)}\in \textrm{span}\{y_{1 (n+1)},..., y_{(n+1) (n+1)}\}, \forall i\in\{1, ..., n\},     \notag\\ 
A y_{(n+1) (n+1)}\in \textrm{span}\{y_{1 (n+1)},..., y_{n  (n+1)}, y_{(n+2) (n+1)}\},     \notag \\
\langle y_{i (n+1)}, y_{(n+2) (n+1)}\rangle =0, \forall i\in\{1, ..., n\},     \notag\\
|\langle y_{(n+1) (n+1)}, y_{(n+2) (n+1)}\rangle|\leq c_{n+1},     \notag\\
\int_0^{T_n-3t_1} \big(S(s)y_{i (n+1)}\big)_x^2(0)\,ds\leq c_{n+1}, \forall i\in\{1, ..., n+2\},     \notag
\end{gather}
this closes the induction loop as the above conditions have the same form of conditions \eqref{ayiinspan}--\eqref{ssyin}.
\\

{\bf{Step 5}}: Find the parameters.
Let $T\geq \big(3B(L, K)-1\big)t_1(K)$.  We find 
\begin{equation}
0< \varepsilon_0\ll \delta_1\ll \delta_2\ll...\ll \delta_{B(L, K)+1}\leq \min\{1/2, t_1\}     
\end{equation}  such that the increasing sequence $\{\calC_n\}$,
\begin{equation}
\mathcal{C}_1= \left(\delta_1^2\tilde{K}^2+ \frac{\varepsilon_0}{\delta_1^2}\right)\frac{24}{\gamma^2},     \notag
\end{equation}
\begin{equation}
\calC_2= \left(\frac{2}{\gamma}\right)^2\left( 24\delta_2^2\tilde{K}^2+ \frac{192\calC_1}{\delta_2^2}+ 16 \tilde{K}^2\calC_1\right),     \notag
\end{equation}
\begin{equation}
\calC_{n+1}=  \left(\frac{2}{\gamma}\right)^2(n+1) \left(6\tilde{K}^2  \delta_{n+1}^2+ \frac{12\calC_n}{\delta_{n+1}^2} +4\tilde{K}^2\calC_n\right), \forall  2\leq n \leq B(L, K),     \notag
\end{equation}
satisfies
\begin{equation}
\frac{3}{2}\sqrt{\frac{(n+1)\calC_n}{t_1}}< \frac{\gamma}{\sqrt{B(L, K)}}, \forall 1\leq n\leq  B(L, K),     \notag
\end{equation}
\begin{equation}
(n+1)\calC_n\leq \min\{1/18, 1/2n\}, 1\leq n\leq  B(L, K).          \notag
\end{equation}
As $\calC_n$ is increasing, it suffices to let 
\begin{equation}
\calC_{B(L, K)}\leq \frac{1}{2B(L,K)} \textrm{ and }\frac{3}{2}\sqrt{\frac{(B(L, K)+1)\calC_{B(L, K)}}{t_1}}< \frac{\gamma}{\sqrt{B(L, K)}}. 
\end{equation}
Suppose that the preceding conditions are fulfilled, then clearly we have $\calC_n\leq \delta_{n+1}^2$.    For ease of  computation, we assume for that moment $\calC_n\leq \delta_{n+1}^2$  and define a sequence $\mathcal{D}_n$ which is larger than $\calC_n$:
\begin{gather}
\mathcal{D}_0= \varepsilon_0, \;\mathcal{D}_1=  \frac{24}{\gamma^2}\left(\delta_1^2\tilde{K}^2+ \frac{\varepsilon_0}{\delta_1^2}\right), \; \mathcal{D}_2= \frac{768}{\gamma^2} \left(\delta_2^2\tilde{K}^2+ \frac{\mathcal{D}_1}{\delta_2^2}\right),     \notag\\
\mathcal{D}_{n+1}= \frac{48(n+1) }{\gamma^2} \left(\tilde{K}^2  \delta_{n+1}^2+ \frac{\mathcal{D}_n}{\delta_{n+1}^2}\right), \forall n\geq 2. \notag
\end{gather}
It suffices to let
\begin{equation}\label{DN}
\frac{3}{2}\sqrt{\frac{(B(L, K)+1)\mathcal{D}_{B(L, K)}}{t_1}}< \frac{\gamma}{\sqrt{B(L, K)}}.
\end{equation}
We try to find $\delta_n$ from backward.  It is rather easy to fix a constant, as $\tilde{\mathcal{D}}_{B(L, K)}$,  that verifies \eqref{DN}. Then we choose $\delta_{n+1}$ and $\tilde{\mathcal{D}}_{n}$ iteratively by making $\tilde{K}^2\delta^2_{n+1}$ and $\tilde{\mathcal{D}}_n/\delta^2_{n+1}$ equivalent, 
 \[\tilde{K}^2\delta^2_{n+1}=  \frac{\tilde{\mathcal{D}}_n}{\delta^2_{n+1}}=  \frac{\gamma^2}{96(n+1)}\tilde{\mathcal{D}}_{n+1}, \forall n\geq 2,  \]
as well as several similar relations for $n=0$ and $1$.  Therefore, we conclude that
\begin{equation}
\delta_{n+1}= \left(\frac{\tilde{D}_n}{\tilde{K}^2}\right)^{\frac{1}{4}}, n\geq 0.\notag
\end{equation}
\begin{equation}
\tilde{D}_{n+1}= \frac{96(n+1)\tilde{K}}{\gamma^2} \left(\tilde{D}_n\right)^{\frac{1}{2}}, n\geq 2; \; \tilde{D}_{2}= \frac{1536\tilde{K}}{\gamma^2} \left(\tilde{D}_1\right)^{\frac{1}{2}},\; \tilde{D}_{1}= \frac{48\tilde{K}}{\gamma^2} \left(\tilde{D}_0\right)^{\frac{1}{2}},\notag
\end{equation}
which gives the values of $\tilde{\mathcal{D}}_0= \varepsilon_0= \varepsilon_0(L,  \gamma, K)$:
\begin{equation}
\varepsilon_0= \left(\tilde{D}_{B(L, K)}\right)^{2^{B(L, K)}} \left(\prod_{k=2}^{B(L, K)-1} \left(\frac{96(k+1)\tilde{K}}{\gamma^2}\right)^{-2^{k+1}}\right) \left(\frac{1536\tilde{K}}{\gamma^2}\right)^{-4} \left(\frac{48\tilde{K}}{\gamma^2}\right)^{-2},
\end{equation}
as well as the values of $\delta_n$ and $\tilde{\mathcal{D}}_n$ that verifies all the above conditions, the details of which we omitted.

To conclude the proof of Proposition \ref{prop:2}, it suffices to take $K_0= 2F^3_s$, $K= K_0$,  $\bar{K}_1(L)= K_1(L, K), T_0(K, L)=3B(L, K)-1$, and $\varepsilon=\varepsilon(L, \gamma):= \varepsilon_0(L, \gamma, K)$.   We say that $\varepsilon_0$ is the value of $\varepsilon$ for Proposition \ref{prop:2}. Indeed, suppose that procedure does not stop for $1\leq n\leq B(L, K)$,  therefore, we have constructed orthonormal functions $\{y_{i B(L, K)+1}\}_{1\leq i\leq B(L, K)+1}\subset \mathcal{A}$, which is in contradiction with Corollary \ref{cor:boundon steps}.   It means that the procedure has to stop at a certain step, $i.e.$ there exists $1\leq m\leq B(L, K)$ such that, we have found orthonormal functions $\{y_{i m}\}_{1\leq i\leq m}$ and function $y^{\perp}_{(m+1) m}$ satisfying $\|y^{\perp}_{(m+1) m}\|<\frac{\gamma}{2}$, then $\{y_{i m}\}_{1\leq i\leq m}$ verifies conditions \eqref{cond-1}--\eqref{cond-5}. It means that $\mathcal{B}_{\gamma}(K_1)$ is not empty. 

\end{proof}

\section{Length critical cases}\label{sec-critical}
Our method also gives  the value of observability constant  for the case $L\in \mathcal{N}$.   The subspace $H$ is called the  controllable  part, thanks to the observability inequality:
\begin{equation}
\int_0^T|\big(S(t)u\big)_x(0)|^2\,dt\geq c\|u\|_{L^2(0,L)}^2, \forall u\in H.  \notag
\end{equation}
Furthermore, both $H$ and $M$ are  $S(t)$ invariant.  Therefore, if we  replace $L^2$ space by $(H, \|\cdot\|_{L^2})$, then the same results  hold, which yields a value of $c(L)$.

\begin{prop}\label{prop:1:H} Let $K_1\geq 1$, $L\in\mathcal{N}$. There exists $\gamma = \gamma(L, K_1)>0$ effectively computable such that the set  $\mathcal{B}_{\gamma}(K_1)$,
\begin{align*}
\mathcal{B}_{\gamma}: = \big\{u\in H^3&(0,L; \mathbb{C})\cap \mathbb{C}H;\,\|u\|_{L^2} = 1,\,\|u\|_{H^3}\leq K_1,\,\,u(0) = u(L) = u_x(L) = 0,\\&|u_x(0)|<\gamma,\;\inf_{\lambda\in \mathbb{C}}\|\lambda u - u_x - u_{xxx}\|_{L^2}<\gamma\big\}
\end{align*}
is empty.
\end{prop}
\begin{prop}\label{prop:2:H} 
There exist $\bar{K}_1(L)$ and $T_0(L)$ such that for any $\gamma>0$ there is $\epsilon = \epsilon(L, \gamma)>0$ effectively computable with the property that, if there are $u\in H\backslash\{0\},  K_1\geq \bar{K}_1(L)$,  and $T\geq T_0(L)$  satisfying
\begin{equation}
\int_0^T \big|\big(S(t)u\big)_x(t,0)\big|^2\,\,dt<\epsilon \|u\|_{L^2(0,L)}^2, \notag
\end{equation}
then $\mathcal{B}_{\gamma}$ is not empty. 
\end{prop}
Let us comment on Proposition \ref{prop:1:H}.  Define a set of eigenfunctions
\[S_L:=\{\lambda; \lambda u= Au, u\in M\},\]
then following the proof of Proposition \ref{prop:1} we are able to give some computable and small $\gamma$ such that,  if $u\in \mathcal{B}_{\gamma}$ then we can find some $u_1$ and some $\lambda\in S_L$ verifying
\[\,u_1(0) = u_1(L) = (u_1)_x(L)=(u_1)_x(0)=0,\,\|\lambda u_1 - (u_1)_x - (u_1)_{xxx}\|_{L^2}<\gamma_1,\]
where $\gamma_1= \gamma_1(\gamma)$ can be sufficiently small if $\gamma$ is.
This can be considered as perturbation of $M$, thus contradicts the fact that $u\in H$ by assuming $\gamma$ small than a certain  computable value.

\section{Further comments and questions}
This is a quantitative way of characterizing observability constant,   mainly based on  flux observations and   strong smoothing effects of the initial boundary value problem, $e.g.$ Observations $(i)$--$(vii)$ and Lemma \ref{thm-flow-smooth} in our case.  We believe that this method can be  applied to many other models.  Moreover, it is also of great interests to consider the following further questions.
\\

\noindent$\bullet$  \textbf{Observability constant behavior around critical points}\\
Let $L_0\in \mathcal{N}$.   Our method gives a finite constant $c(L_0)>0$, while it also provides a vanishing sequence $\{c(L)\}_{L\rightarrow L_0}$:  $c(L)\rightarrow 0^+$.  Apparently, this difference, $e.g.$ the ``jump'' of the observability, comes from the uncontrollable subspace $M$.    Thus when the length is not critical it is natural to ask for the existence of subspace, comparing to $M$, such that the observability constant of the quotient space is continuous.  In other words, is that possible to find  some finite dimensional space $M_L$ for $L$ near $L_0$ such that the related ``observability'' of $L^2/M_L$, which is denoted  by $c_{H_L}(L)$, satisfies  $c_{H_L}(L)\rightarrow c(L_0)$.
\\

\noindent$\bullet$  \textbf{Optimal estimates}\\
According to  the duality between controllability and observability,  the sharp observability inequality constant is also the  optimal control cost.    This optimal estimate is of both mathematical and engineering interests, as stated in Introduction, it is the fundamental result for many other studies upon this model.  However, it does not seem that the value that we obtain in this paper is optimal. Therefore, it would be interesting to further get sharp estimates of $c(L)$. 
\\

\noindent$\bullet$  \textbf{Observability inequality for small time}\\
On account of  to the smoothing procedure $S(t_1(L, K_1))$ and Lemma \ref{lem:compact}, our constructive approach and quantitative result only apply for large time, $i.e.$ $T$ bigger than some  $T_0(L, K_1)$.  It is not clear whether some modifications and minimizations on our method can make the time small.   What is the behavior of the constant  when $T$ tends to $0$, and what is the sharp asymptotic estimate? Should the cost be like  $Ce^{-\frac{c}{T^{\alpha}}}$,  as it is the case for many other models \cite{Zuazua-cost-2011, Lebeau-sharp-2016}?  
\\ 

\noindent$\bullet$  \textbf{Is backstepping another option?}\\
Originally introduced to  stabilize system exponentially \cite{coron-1991-adding-integral, 2008-Krstic-Smyshlyaev-book}, recently it is further developed as a tool for  null control and small-time stabilization problems \cite{2017-Coron-Nguyen-ARMA, 2019-xiang-SICON, 2017-Xiang-SCL, Coron-Xiang-2018, ZhangRapidStab}, the so called piecewise backstepping, which shares the advantage that the feedback (control) is well formulated.  It consists in stabilizing system with arbitrary exponential decay rate  (rapid stabilization) with explicit computable estimates,  and splitting the time interval into infinite many parts such that on each part backstepping exponential stabilization is applied to make the energy  divide at least by 2.     Concerning our KdV case, at least for non-critical cases,  rapid stabilization by  backstepping  is achieved in  \cite{coronluqi},  where they used the controllability of KdV equation with  control of the form $b(t)=u_x(t, 0)-u_x(t, L)$  as an intermediate step.    If it is possible to perform piecewise backstepping by obtaining $Ce^{c \lambda^{\alpha}}$ type estimates on each step, then we are able to get null controllability and small time stabilization with precise cost estimates.

\appendix
\section{The $L\leq 4$ case.}\label{appL4}
Let us consider the flow $y(t)=S(t)y_0$.  Integration by parts and \eqref{fxbound} \eqref{fluxine} show
\begin{gather}
T\int_0^L y_0^2(x)\,dx \leq\int_0^T\int_0^L y^2(t, x) \,dx\,dt+ T\int_0^Ty_x^2(t, 0)\,dt,     \notag
\end{gather}
then Poincare's inequality lead to 
\begin{gather}
T\int_0^L y_0^2(x)\,dx \leq\frac{L^2}{\pi^2}\int_0^T\int_0^L y_x^2(t, x) \,dx\,dt+ T\int_0^Ty_x^2(t, 0)\,dt,     \notag
\end{gather}    
 which combines with \eqref{fxbound} yield 
 \begin{equation}
 \left(T-\frac{L^2(T+L)}{3\pi^2}\right)\int_0^L y_0^2(x)\,dx \leq T\int_0^Ty_x^2(t, 0)\,dt.     \notag
 \end{equation}
Consequently, when $T$ and $L$ satisifes $\frac{L^3}{3T\pi^2}+\frac{L^2}{3\pi^2}<1$,  the observability constant can be  $\frac{3T\pi^2}{3T\pi^2-L^3-TL^2}$.  

\section{Sobolev estimates and some properties of the flow}\label{app-est}
We start from giving some quantitative   Sobolev embedding and interpolation estimates.  In the literature these bounds are usually simply provided by some unknown constant $C$, for example Brezis \cite{Brezis-book} and Adams \cite{Adams-book}, though ideas of getting which are well illustrated.    

For any $\xi\in (0, L/3), \eta\in (2L/3, L)$, there exists $\lambda\in (\xi, \eta)$ such that 
\begin{equation}
|f'(\lambda)|= \left|\frac{f(\eta)-f(\xi)}{\eta-\xi}\right|\leq \frac{3}{L}\left(|f(\eta)|+|f(\xi)|\right). \notag
\end{equation}
Therefore, $\forall x\in (0, L)$, 
\begin{equation}
|f'(x)|= \left|f'(\lambda)+ \int_{\lambda}^x f''(t)\,dt\right|\leq  \frac{3}{L}\left(|f(\eta)|+|f(\xi)|\right)+ \int_0^L |f''(t)| \,dt,  \notag
\end{equation}
then we integrate $\xi$ on $(0, L/3)$ and $\eta$ on $(2L/3, L)$ to get 
\begin{equation}
|f'(x)|\leq \frac{9}{L^2}\int_0^L |f(t)| \,dt+ \int_0^L |f''(t)| \,dt.  \notag
\end{equation}
Hence,
\begin{equation}\label{120-L}
\int_0^L |f'(x)|^2 \,dx\leq \frac{162}{L^2}\int_0^L |f(t)|^2 \,dt+ 2L^2\int_0^L |f''(t)|^2 \,dt. 
\end{equation}
Because  for any $\delta\in (0, L^2]$ there exists $n\in \mathbb{N}$ such that $L/n\in [\delta^{1/2}/2, \delta^{1/2}]$,  we can split $[0, L]$ by $n$ parts.  By performing \eqref{120-L} on each part and combining them together, we get
\begin{equation}\label{delta--120}
\int_0^L |f'(x)|^2 \,dx\leq  2\delta\int_0^L |f''(t)|^2 \,dt+ \frac{648}{\delta}\int_0^L |f(t)|^2 \,dt, \; \forall \delta\in (0, 16], \notag
\end{equation}
thus
\begin{equation}\label{delta120}
\int_0^L |f'(x)|^2 \,dx\leq  42\left(\delta\int_0^L |f''(t)|^2 \,dt+ \frac{1}{\delta}\int_0^L |f(t)|^2 \,dt\right), \; \forall \delta\in (0, 1].
\end{equation}
Notice that \eqref{delta120} also holds for  complex valued functions.  By replacing $f$ by $f^{(n)}$, we also get
\begin{equation}\label{ndelta120}
\int_0^L |f^{(n+1)}(x)|^2 \,dx\leq  42\left(\delta\int_0^L |f^{(n+2)}(t)|^2 \,dt+ \frac{1}{\delta}\int_0^L |f^{(n)}(t)|^2 \,dt\right), \; \forall \delta\in (0, 1]. 
\end{equation}
Moreover,  we are able to find a constant $E_m^n$ which only depends on $m$ and $n$ such that 
\begin{equation}\label{ineemb}
\int_0^L |f^{(n)}(x)|^2 \,dx\leq  E_m^n\left(\delta^{m-n}\int_0^L |f^{(m)}(t)|^2 \,dt+ \delta^{-n}\int_0^L |f(t)|^2 \,dt\right), \; \forall \delta\in (0, 1],
\end{equation}
while, more precisely,  $E^n_m$ can be calculated by
\begin{gather}
E^1_2= 42,\\
E^m_{m+1}= 2^{m}  42^m (E^{m-1}_{m})^{m}, \label{Emm1}\\
E^{k-1}_m= E^{k-1}_k(E^k_m+1). \label{Ek1m}
\end{gather}
For ease of notations, we denote 
\begin{equation}
a_n:= \int_0^L |f^{(n)}(x)|^2 \,dx= \lVert f\lVert_{\dot{H}^n}^2 \textrm{ and }   \lVert f\lVert_{H^n}^2= \lVert f\lVert_{\dot{H}^n}^2+ \lVert f\lVert_{L^2}^2.\notag
\end{equation}
In fact, $E^1_2= 42$ as shown in \eqref{delta120}, further estimated are obtained from a reduction procedure on $m$.  
Suppose that $E^n_i$ with $i\leq m$ are known, then from \eqref{ndelta120} and \eqref{ineemb} we derive
\begin{gather}
a_{m-1}\leq E^{m-1}_m\left(\delta_1 a_m+ \delta_1^{-(m-1)}a_0 \right),  \; \forall \delta_1\in (0, 1], \notag\\
a_{m}\leq E^{1}_2\left(\delta a_{m+1}+ \delta^{-1}a_{m-1} \right),  \; \forall \delta\in (0, 1]. \notag
\end{gather}
By taking  $\delta_1:= \delta/(2E_2^1 E_m^{m-1})$, we obtain
\begin{equation}
a_m\leq 2 E_2^1 \left(\delta a_{m+1}+ 2^{m-1} (E^1_2)^{m-1}(E_m^{m-1})^{m} \delta^{-m} a_0 \right), \notag
\end{equation}
which concludes \eqref{Emm1}. As for \eqref{Ek1m},  we perform \eqref{ndelta120} and \eqref{ineemb} once again to get,  for $k\leq m$, 
\begin{align*}
a_{k-1}&\leq E_k^{k-1}\left(\delta a_k+ \delta^{-(k-1)}a_0\right),\\
&\leq E_k^{k-1}\left(\delta E^k_{m+1}(\delta^{m+1-k}a_{m+1}+ \delta^{-k}a_0)+ \delta^{-(k-1)}a_0\right),\\
&\leq E_k^{k-1}(E_{m+1}^k+1)\left(\delta^{m+2-k}a_{m+1}+ \delta^{-(k-1)}a_0\right).
\end{align*} 

By taking $\delta:= (a_0/(a_0+a_m))^{1/m}$ in \eqref{ndelta120}, we get 
\begin{align*}
a_n&\leq E^n_m \left((\frac{a_0}{a_0+a_m})^{\frac{m-n}{m}}a_m+(\frac{a_0+a_m}{a_0})^{\frac{n}{m}}a_0\right),\\
&\leq 2E^n_m a_0^{\frac{m-n}{m}} (a_0+a_m)^{\frac{n}{m}}.
\end{align*}
This implies that 
\begin{equation}
\lVert f\lVert_{\dot{H}^n}^2\leq 2E^n_m\lVert f\lVert_{L^2}^{\frac{2(m-n)}{m}} \lVert f\lVert_{H^m}^{\frac{2n}{m}},     \notag
\end{equation}
thus
\begin{equation}
\lVert f\lVert_{H^n}^2\leq (2E^n_m+1)\lVert f\lVert_{L^2}^{\frac{2(m-n)}{m}} \lVert f\lVert_{H^m}^{\frac{2n}{m}}, \forall\; 0<n<m.     \notag
\end{equation}
\linebreak

Now we turn to the proof of Lemma \ref{lem:compact} and Corollary \ref{cor:boundon steps} .   Actually, assuming  Lemma \ref{lem:compact}, for any $g_i$ there exists $f_{n_i}$ such that 
\[ 
\|g_i-f_{n_i}\|_{L^2}<\frac{\sqrt{2}}{2}.  
\]
Suppose that $P>R$, then there exists $i\neq j$ such that $f_{n_i}=f_{n_j}$, contradiction, as 
\[ 
\sqrt{2}= \|g_i-g_j\|_{L^2}\leq   \|g_i-f_{n_i}\|_{L^2}+ \|g_j-f_{n_i}\|_{L^2}< \sqrt{2}.
\]
It remains to prove Lemma \ref{lem:compact} which is of course a direct consequence of Rellich's theorem.  In fact, as the injection $H^3\hookrightarrow L^2$ is compact, it suffices to find a finite open cover  composed by the union of balls with radius $\sqrt{2}/2$.   By this way,  $f_i$ can be chosen in  $\mathcal{A}$. However, one does not know the exact value of covering balls.  Instead, we present a constructive proof, which explicitly characterize the value of  $B$. 

 Notice that if $f\in \mathcal{A}$ then $f$ satisfies, $f\in H$ and $f(0)=f(L)=0$, which means  
\[
f= \sum_{n\in \mathbb{N}^*} a_n \sin \left( \frac{n\pi x}{L}\right) \textrm{ in } H^1,
\]
with its $H^1$ norm  given by 
\begin{equation}
\|f\|_{\dot{H}^1}^2= \sum_{n\in \mathbb{N}^*} a_n^2  \frac{n^2\pi^2}{2L}, \;\|f\|_{L^2}^2= \sum_{n\in \mathbb{N}^*} a_n^2  \frac{L}{2}.     \notag
\end{equation}
Thanks to \eqref{ineemb} and the definition of $\mathcal{A}$,
\begin{equation}
\|f\|_{\dot{H}^1}^2\leq E^1_3 \left(\int_0^L \|f^{(3)}(x)\|^2+\|f(x)\|^2 \,dx\right)= E^1_3 \|f\|_{H^3}^2,     \notag
\end{equation}
then 
\begin{equation}\label{anH1bound}
 \sum_{n\in \mathbb{N}^*} a_n^2  \frac{n^2\pi^2}{2L}\leq E^1_3 K^2,
\end{equation}
thus 
\begin{equation}
 a_n\leq \frac{K \sqrt{2L E^1_3}}{n\pi}.     \notag
\end{equation}

Next, we pick up all the functions of the following form, which are denoted by $\{f_m\}$, 
\begin{gather}
f_m= \sum_{n=1}^{N_c-1} a_n^m \sin \left( \frac{n\pi x}{L}\right),     \notag\\
|a_n^m|\in \left\lbrace0, \frac{K \sqrt{2L E^1_3}}{n\pi}\cdot \frac{1}{M_c}, \frac{K \sqrt{2L E^1_3}}{n\pi}\cdot \frac{2}{M_c}, ..., \frac{K \sqrt{2L E^1_3}}{n\pi} \cdot \frac{M_c}{M_c}\right\rbrace,     \notag
\end{gather}
where $N_c$ and $M_c$ are some integers only depend on $L$  to be chosen later on. 

It can be proved that with a good choice of $N_c$ and $M_c$ the above sequence $\{f_m\}$ satisfies Lemma \ref{lem:compact}.  Clearly, $f_m\in C^{\infty}\subset H^3([0, L])$.  Let $f\in \mathcal{A}$. On the one hand, thanks to the above construction, there exists a function $f_m$ such that 
\begin{equation}
|a_n^m-a_n|< \frac{K \sqrt{2L E^1_3}}{n\pi}\cdot \frac{1}{2M_c}, \forall n\in \{1, 2,..., N_c-1\}.     \notag
\end{equation}
Hence 
\begin{align*}
\sum_{n=1}^{N_c-1} (a_n^m-a_n)^2\cdot \frac{L}{2}< \frac{L}{2} \sum_{n=1}^{N_c-1} \frac{LE^1_3 K^2}{2M_c^2 n^2\pi^2}\leq \frac{L^2}{4}\cdot \frac{1}{6}\cdot \frac{E^1_3 K^2}{M_c^2}.
\end{align*}
On the other hand, we know from \eqref{anH1bound} that
\begin{equation}
\sum_{n\geq N_c} (a_n^m-a_n)^2\cdot \frac{L}{2}\leq E^1_3 K^2 \frac{L^2}{N_c^2}.     \notag
\end{equation}
Therefore, we can choose $M_c$ and $N_c$ as
\begin{equation}
M_c= M_c(L)= \Big{\lceil} KL\sqrt{\frac{E^1_3}{6}}\Big{\rceil}, \; N_c= N_c(L)=\Big{\lceil} 2KL\sqrt{E^1_3}\Big{\rceil},     \notag
\end{equation}
which yields
\begin{equation}
\|f_m-f\|_{L^2}^2= \sum_{n\in \mathbb{N}^*} (a_n-a_n^m)^2  \frac{L}{2}< \frac{1}{2}.     \notag
\end{equation}
In such a case, the value of $B(L, K)$ is given by $(2M_c+1)^{N_c-1}$.
\linebreak

\begin{proof}[Proof of Lemma \ref{thm-flow-con}]
$(i)$.  \textit{Case $k=0$}. 
It is a  straightforward consequence of \eqref{fxbound}--\eqref{fluxine} that 
\begin{equation}
F_0^0= 1, \; F_1^0= \sqrt{5L/3}.     \notag
\end{equation}

$(ii)$. \textit{Case $k=3$}. Suppose that $f_0\in L^2$, then as $f$ satisfies
\begin{gather*}
f_t(t, x)= Af(t, x),  t\in (0, T), x\in (0, L) \\
f(t, 0)=f(t, L)=f_x(t, L)=0,  t\in (0, T)\\
f(0, x)= f_0(x), x\in (0, L),
\end{gather*} 
we know that $g:= f_t= Af$ is the solution of 
\begin{gather*}
g_t(t, x)= Ag(t, x), \; t\in (0, T), x\in (0, L) \\
g(t, 0)=g(t, L)=g_x(t, L)=0,  \;  t\in (0, T)\\
g(0, x)= g_0(x):= (Af_0)(x),\;   x\in (0, L).
\end{gather*} 
Since
\[
\|u_{x}\|_{L^2}^2\leq E^1_3(\delta^2\|u_{xxx}\|_{L^2}^2+ \delta^{-1}\|u\|_{L^2}^2),
\]
we have
\[
\|u_{x}\|_{L^2}\leq \sqrt{E^1_3}(\delta\|u_{xxx}\|_{L^2}+ \delta^{-1/2}\|u\|_{L^2}).
\]
Therefore, by choosing $\delta:= 1/\sqrt{4 E^1_3}$ we get
\[
\|u_{x}\|_{L^2}\leq \frac{1}{2}\|u_{xxx}\|_{L^2}+(4 E^1_3)^{3/4}\|u\|_{L^2},
\]
which implies
\begin{align*}
\|u_{xxx}\|_{L^2}&\leq  \|Au\|_{L^2}+ \|u_x\|_{L^2},\\
&\leq \|Au\|_{L^2}+ \frac{1}{2}\|u_{xxx}\|_{L^2}+(4 E^1_3)^{3/4}\|u\|_{L^2},\\
&\leq 2\|Au\|_{L^2}+ 2 (4 E^1_3)^{3/4}\|u\|_{L^2}
\end{align*}
and
\begin{equation}
\|Au\|_{L^2}\leq \|u_{xxx}\|_{L^2}+  \|u_x\|_{L^2}\leq (1+\sqrt{E^1_3})\|u\|_{H^3}.
\end{equation}

Thanks to  the result in the case k=0, we have
\begin{equation}
\|f_{x}\|_{L^2(0, T; L^2)}\leq \sqrt{\frac{2L}{3}}\|f_0\|_{L^2}, \; \|f(t)\|_{L^2}\leq \|f_0\|_{L^2}     \notag
\end{equation}
and, by replacing $f$ by $g$,
\begin{gather*}
\int_0^T\int_0^L g_x^2(t, x)\,dx \,dt\leq  \frac{2L}{3}\int_0^L g_0^2(x)\,dx, \\
\int_0^L g^2(t, x)\,dx\leq \int_0^L g_0^2(x)\,dx,
\end{gather*}
which implies
\begin{gather*}
\|Af(t)\|_{L^2}\leq  \|Af_0\|_{L^2}\leq (1+ \sqrt{E^1_3})\|f_0\|_{H^3}, \\
\|A(f_x)\|_{L^2(0, T; L^2)}= \|(Af)_x\|_{L^2(0, T; L^2)}\leq \sqrt{\frac{2L}{3}}(1+ \sqrt{E^1_3})\|f_0\|_{H^3}.
\end{gather*}
Hence,
\begin{align*}
\|f_{xxx}(t)\|_{L^2}&\leq 2(1+ \sqrt{E^1_3})\|f_0\|_{H^3}+ 2 (4 E^1_3)^{3/4}\|f(t)\|_{L^2},\\
&\leq 2(1+ \sqrt{E^1_3})\|f_0\|_{H^3}+ 2 (4 E^1_3)^{3/4}\|f_0\|_{L^2}, \\
\|f_{x}(t)\|_{L^2}&\leq  \frac{1}{2}\|f_{xxx}(t)\|_{L^2}+(4 E^1_3)^{3/4}\|f(t)\|_{L^2}, \\
&\leq (1+ \sqrt{E^1_3})\|f_0\|_{H^3}+ 2 (4 E^1_3)^{3/4}\|f_0\|_{L^2}, \\
\|f_{xxxx}\|_{L^2(0, T; L^2)}&\leq 2\|A(f_x)\|_{L^2(0, T; L^2)}+ 2 (4 E^1_3)^{3/4}\|f_x\|_{L^2(0, T; L^2)}, \\
&\leq 2\sqrt{\frac{2L}{3}}(1+ \sqrt{E^1_3})\|f_0\|_{H^3}+ 2 \sqrt{\frac{2L}{3}} (4 E^1_3)^{3/4}\|f_0\|_{L^2},
\end{align*}
thus
\begin{align*}
\|S(t)f_0\|_{C([0, T]; H_{(0)}^3(0, L))}&\leq \left(2(1+ \sqrt{E^1_3})+ 2 (4 E^1_3)^{3/4}+1\right) \|f_0\|_{H_{(0)}^3(0, L)},\\
\|S(t)f_0\|_{L^2(0, T; H_{(0)}^{4}(0, L))}&\leq    \left(2\sqrt{\frac{2L}{3}}(1+ \sqrt{E^1_3})+ 2 \sqrt{\frac{2L}{3}} (4 E^1_3)^{3/4}+ \sqrt{L}\right)\|f\|_{H_{(0)}^3},
\end{align*}
which gives the value of $F^3_i$:
\begin{align}
F^3_0&= 2(1+ \sqrt{E^1_3})+ 2 (4 E^1_3)^{3/4}+1,\\
F^3_1&= 2\sqrt{\frac{2L}{3}}(1+ \sqrt{E^1_3})+ 2 \sqrt{\frac{2L}{3}} (4 E^1_3)^{3/4}+ \sqrt{L}.
\end{align}

$(iii)$.  \textit{Case $k=6$}.   Suppose that $f_0\in H_{(0)}^{6}(0, L)$, then  $g:=f_t=Af$ satisfies 
\begin{gather*}
g_t(t, x)= Ag(t, x), \; t\in (0, T), x\in (0, L) \\
g(t, 0)=g(t, L)=g_x(t, L)=0,  \;  t\in (0, T)\\
g(0, x)= g_0(x):= (Af_0)(x),\;   x\in (0, L),
\end{gather*} 
and $h:=g_t=Ag=A^2 f$ satisfies
\begin{gather*}
h_t(t, x)= Ah(t, x), \; t\in (0, T), x\in (0, L) \\
h(t, 0)=h(t, L)=h_x(t, L)=0,  \;  t\in (0, T)\\
h(0, x)= h_0(x):= (Af_0)(x)= (A^2f_0)(x),\;   x\in (0, L).
\end{gather*} 
Simple embedding estimate shows
\begin{gather}
\|u^{(4)}\|_{L^2}\leq \frac{1}{4}\|u^{(6)}\|_{L^2}+ 16(E^4_6)^{3/2}\|u\|_{L^2},\; \|u^{(4)}\|_{L^2}\leq \sqrt{E^4_6} \|u\|_{H^6}, \notag\\
\|u^{(2)}\|_{L^2}\leq \frac{1}{4}\|u^{(6)}\|_{L^2}+ 2(E^2_6)^{3/4}\|u\|_{L^2},\; \|u^{(2)}\|_{L^2}\leq  \sqrt{E^2_6} \|u\|_{H^6}. \notag
\end{gather}
It is known from the case $k=0$ that
\begin{gather}
\|f^{(6)}(t)+2f^{(4)}(t)+f^{(2)}(t)\|_{L^2}= \|h(t)\|_{L^2}\leq \|h_0\|_{L^2},     \notag\\
\|f^{(7)}+2f^{(5)}+f^{(3)}\|_{L^2(0, T; L^2)}= \|h_x\|_{L^2(0, T; L^2)}\leq \sqrt{\frac{2L}{3}}\|h_0\|_{L^2},     \notag
\end{gather}
which, combined with the preceding embedding estimates, yields
\begin{align*}
\|f^{(6)}(t)\|_{L^2}&\leq 4\|h_0\|_{L^2}+ 128(E^4_6)^{3/2}\|f\|_{L^2}+ 8(E^2_6)^{3/4}\|f\|_{L^2},\\
&\leq \left(8\sqrt{E^4_6}+ 4\sqrt{E^2_6}+ 128(E^4_6)^{3/2}+ 8(E^2_6)^{3/4}\right)\|f_0\|_{H^6},
\end{align*}
and
\begin{align*}
\|f^{(7)}\|_{L^2(0, T; L^2)}&\leq 4\|h_x\|_{L^2(0, T; L^2)}+ 128(E^4_6)^{3/2}\|f_x\|_{L^2(0, T; L^2)}+ 8(E^2_6)^{3/4}\|f_x\|_{L^2(0, T; L^2)},\\
&\leq 4\sqrt{\frac{2L}{3}} \|h_0\|_{L^2}+ \sqrt{\frac{2L}{3}}\left(128(E^4_6)^{3/2}+ 8(E^2_6)^{3/4}\right)\|f_0\|_{L^2},\\
&\leq \left(8\sqrt{\frac{2L}{3}}(E^4_6)^{1/2}+ 4\sqrt{\frac{2L}{3}}(E^2_6)^{1/2}+ 128\sqrt{\frac{2L}{3}}(E^4_6)^{3/2}+  8\sqrt{\frac{2L}{3}}(E^2_6)^{3/4}\right).
\end{align*}
Thus, the value of $F^6_i$ can be chosen as
\begin{align*}
F^6_0&= 8\sqrt{E^4_6}+ 4\sqrt{E^2_6}+ 128(E^4_6)^{3/2}+ 8(E^2_6)^{3/4}+1,\\
F^6_1&= 8\sqrt{\frac{2L}{3}}(E^4_6)^{1/2}+ 4\sqrt{\frac{2L}{3}}(E^2_6)^{1/2}+ 128\sqrt{\frac{2L}{3}}(E^4_6)^{3/2}+  8\sqrt{\frac{2L}{3}}(E^2_6)^{3/4}+ \sqrt{L}.
\end{align*}

$(iii)$.  \textit{Case $k=1, 2, 4,5$}.
It can be achived by  the (real) interpolation of Sobolev spaces.  To avoiding getting too much involved into this classical theory, we directly use some quantitative results in \cite{sobolev-quantitative}, and following several related notations there.
\begin{equation}
\|u\|_{\mathcal{H}^m(\mathbb{R})}^2:= \sum_{\alpha\leq m}\binom{m}{\alpha}\|\partial^{\alpha}u\|^2_{L^2(\mathbb{R})},\;
\|u\|_{\mathcal{H}^m(0, L)}^2:= \sum_{\alpha\leq m}\binom{m}{\alpha}\|\partial^{\alpha}u\|^2_{L^2(0, L)}, \notag
\end{equation}
the interpolation spaces as well as their norms are given by  K-method, 
\begin{align*}
\overline{H^1}= \big(\mathcal{H}^0(0, L), \mathcal{H}^3(0, L)\big)_{\frac{1}{3}},&\; \overline{H^2}= \big(\mathcal{H}^0(0, L), \mathcal{H}^3(0, L)\big)_{\frac{2}{3}}, \notag\\
\overline{H^4}= \big(\mathcal{H}^3(0, L), \mathcal{H}^6(0, L)\big)_{\frac{1}{3}},&\; \overline{H^5}= \big(\mathcal{H}^3(0, L), \mathcal{H}^6(0, L)\big)_{\frac{2}{3}},\notag
\end{align*}
{\small
\begin{align*}
 \overline{L^2H^2}:= \big(L^2(0, T;\mathcal{H}^1(0, L)), L^2(0, T; \mathcal{H}^4(0, L)\big)_{\frac{1}{3}},& \; \overline{L^2H^3}:= \big(L^2(0, T;\mathcal{H}^1(0, L)), L^2(0, T; \mathcal{H}^4(0, L)\big)_{\frac{2}{3}}, \notag\\
 \overline{L^2H^5}:= \big(L^2(0, T;\mathcal{H}^4(0, L)), L^2(0, T; \mathcal{H}^7(0, L)\big)_{\frac{1}{3}},& \; \overline{L^2H^6}:= \big(L^2(0, T;\mathcal{H}^4(0, L)), L^2(0, T; \mathcal{H}^7(0, L)\big)_{\frac{2}{3}}.\notag
\end{align*}}
Then we have the following lemma concerning these interpolation spaces.
\begin{lemma}\label{lem-inter-cite}
($\textit{I}$) There exists an extension $\mathcal{E}$ and constants $\lambda_m= \lambda_m(L)$ such that 
\begin{gather}
\mathcal{E}: \mathcal{H}^m(0, L)\rightarrow\mathcal{H}^m(\mathbb{R}), \notag\\
\|u\|_{\mathcal{H}^m(0, L)}\leq \|\mathcal{E}u\|_{\mathcal{H}^m(\mathbb{R})}\leq \lambda_m \|u\|_{\mathcal{H}^m(0, L)}, \forall m\in \{0,1,2, 3,4,5, 6,7\}.\notag
\end{gather}\\($\textit{II}$) The norms $\mathcal{H}^m(0, L)$ and $\overline{H^m}$ are equivalent: 
\begin{gather}\label{interpo-H}
(\lambda^0)^{-\frac{2}{3}}(\lambda^3)^{-\frac{1}{3}}\|u\|_{\mathcal{H}^1(0, L)}\leq \|u\|_{\overline{H^1}}\leq \|u\|_{\mathcal{H}^1(0, L)},
\end{gather}
moreover, 
\begin{gather}\label{interpo-LH}
(\lambda^1)^{-\frac{2}{3}}(\lambda^4)^{-\frac{1}{3}}\|u\|_{L^2(0, T;\mathcal{H}^2(0, L))}\leq \|u\|_{\overline{L^2H^2}}\leq \|u\|_{L^2(0, T;\mathcal{H}^2(0, L))}.
\end{gather}
Similar results hold for $\overline{H^m}$ and $\overline{L^2H^{m+1}}$ when $m\in \{2, 4, 5\}$.\\

\noindent ($\textit{III}$) There exist constants $G^m$ such that
\begin{equation}
\|u\|_{H^m(0, L)}\leq \|u\|_{\mathcal{H}^m(0, L)}\leq G^m\|u\|_{H^m(0, L)}, \forall m\in \{0, 1,2,3,4,5,6, 7\}.\notag
\end{equation}

\end{lemma}

\begin{proof}[Proof of Lemma \ref{lem-inter-cite}]
($I$) This is a classical extension problem, we recall Stein \cite[page 182 Theorem $5'$]{Stein} for a precise construction.  In fact the same type of results also exists for many other spaces, like Besov space $etc$.

($II$) Inequality \eqref{interpo-H} is exactly \cite[Lemma 4.2]{sobolev-quantitative}, and the same method also leads to \eqref{interpo-LH}.

 ($III$) The first inequality is obvious. It suffices to prove the second one.   If $m=0$ or $1$, then $G^m=1$. Else, we get from the definition that 
 \begin{align*}
 \|u\|_{\mathcal{H}^m(0, L)}^2&= \sum_{\alpha\leq m}\binom{m}{\alpha}\|\partial^{\alpha}u\|^2_{L^2(0, L)},\\
 &=  \|u\|_{H^m(0, L)}^2+ \sum_{0<\alpha< m}\binom{m}{\alpha}\|\partial^{\alpha}u\|^2_{L^2(0, L)}, \\
 &\leq  \|u\|_{H^m(0, L)}^2+ \sum_{0<\alpha< m}\binom{m}{\alpha} E^{\alpha}_m\|u\|_{H^m(0, L)}^2, \\
 &=  \|u\|_{H^m(0, L)}^2\left(1+\sum_{0<\alpha< m}\binom{m}{\alpha} E^{\alpha}_m\right),
 \end{align*}
 which gives the value of $G^m$:
 \begin{equation}
 G^m:= 1+\sum_{0<\alpha< m}\binom{m}{\alpha} E^{\alpha}_m.  \notag
 \end{equation}
\end{proof}
Armed with the preceding lemma, we can apply the interpolation theory on cases $k=1,2,4$ and $5$.  Here we only prove the case $k=1$,  while the other cases can be treated in the same way.  Since we are dealing with the KdV flow, we add the natural compatibility conditions on interpolation spaces, for example $\overline{H^1_{(0)}}$ which is endowed with the same norm as $\overline{H^1}$.

For any $t\in (0, T]$, we define a mapping operator 
\begin{equation}
\mathcal{L}_0^t: f\longmapsto S(t)f.\notag
\end{equation}
We also define 
\begin{equation}
\mathcal{L}_1: f\longmapsto S(\cdot)f, t\in [0, T].\notag
\end{equation}
From the preceding part we know that, for $m\in\{0, 3\}$ the linear operators
\begin{gather}
\mathcal{L}_0^t: \mathcal{H}^m_{(0)}(0, L)\rightarrow \mathcal{H}^m_{(0)}(0, L),\notag\\
\mathcal{L}_1: \mathcal{H}^m_{(0)}(0, L)\rightarrow L^2(0, T; \mathcal{H}^{m+1}_{(0)}(0, L)),\notag
\end{gather}
 are bounded. Indeed, these bounded are given by
\begin{gather}
\|\mathcal{L}_0^t\|_{\mathcal{H}^m, \mathcal{H}^m}\leq F_0^m G^m,\;
\|\mathcal{L}_1\|_{\mathcal{H}^m,L^2\mathcal{H}^{m+1}}\leq F_1^{m} G^{m+1}.  \notag
\end{gather}
Therefore, thanks to the interpolation theory, we get 
\begin{equation}
\|\mathcal{L}_0^t\|_{\overline{H^1}, \overline{H^1}}\leq \|\mathcal{L}_0^t\|_{\mathcal{H}^0, \mathcal{H}^0}^{\frac{2}{3}} \|\mathcal{L}_0^t\|_{\mathcal{H}^3, \mathcal{H}^3}^{\frac{1}{3}}\leq (F^0_0G^0)^\frac{2}{3}(F^3_0G^3)^\frac{1}{3}, \notag
\end{equation}
\begin{equation}
\|\mathcal{L}_1\|_{\overline{H^1}, \overline{L^2H^2}}\leq \|\mathcal{L}_1\|_{\mathcal{H}^0, L^2\mathcal{H}^1}^{\frac{2}{3}} \|\mathcal{L}_1\|_{\mathcal{H}^3, L^2\mathcal{H}^4}^{\frac{1}{3}}\leq (F^0_1G^1)^\frac{2}{3}(F^3_1G^4)^\frac{1}{3}. \notag
\end{equation}
Thus
\begin{equation}
\|\mathcal{L}_0^t\|_{\mathcal{H}^1, \mathcal{H}^1}\leq (\lambda^0)^{\frac{2}{3}}(\lambda^3)^{\frac{1}{3}}\|\mathcal{L}_0^t\|_{\overline{H^1}, \overline{H^1}}\leq (\lambda^0 F^0_0G^0)^\frac{2}{3}(\lambda^3F^3_0G^3)^\frac{1}{3},\notag
\end{equation}
\begin{equation}
\|\mathcal{L}_1\|_{\mathcal{H}^1, L^2\mathcal{H}^2}\leq (\lambda^1)^{\frac{2}{3}}(\lambda^4)^{\frac{1}{3}}\|\mathcal{L}_0^t\|_{\overline{H^1}, \overline{L^2H^2}}\leq (\lambda^1 F^0_1G^1)^\frac{2}{3}(\lambda^4F^3_1G^4)^\frac{1}{3}, \notag
\end{equation}
hence
\begin{equation}
\|\mathcal{L}_0^t\|_{H^1, H^1}\leq G^1 \|\mathcal{L}_0^t\|_{\mathcal{H}^1, \mathcal{H}^1}\leq G^1(\lambda^0 F^0_0G^0)^\frac{2}{3}(\lambda^3F^3_0G^3)^\frac{1}{3},\notag
\end{equation}
\begin{equation}
\|\mathcal{L}_1\|_{H^1, L^2H^2}\leq G^1\|\mathcal{L}_1\|_{\mathcal{H}^1, L^2\mathcal{H}^2}\leq G^1(\lambda^1 F^0_1G^1)^\frac{2}{3}(\lambda^4F^3_1G^4)^\frac{1}{3}, \notag
\end{equation}
Hence we get 
\begin{align*}
\|S(t)f_0\|_{L^{\infty}([0, T]; H_{(0)}^1(0, L))}&\leq  F_0^1 \|f\|_{H_{(0)}^1(0, L)}, \\
\|S(t)f_0\|_{L^2(0, T; H_{(0)}^{2}(0, L))}&\leq   F_1^1 \|f\|_{H_{(0)}^1(0, L)},
\end{align*}
with $F^1_0, F^1_1$ defined by 
\begin{equation}
F^1_0:= G^1(\lambda^0 F^0_0G^0)^\frac{2}{3}(\lambda^3F^3_0G^3)^\frac{1}{3}, \;F^1_1:= G^1(\lambda^1 F^0_1G^1)^\frac{2}{3}(\lambda^4F^3_1G^4)^\frac{1}{3}.\notag
\end{equation}
As the flow conserves the Sobolev regularity, we know that 
\begin{equation}
\|S(t)f_0\|_{C^0([0, T]; H_{(0)}^1(0, L))}\leq  F_0^1 \|f\|_{H_{(0)}^1(0, L)}.
\end{equation}

Similar calculation provides
\begin{align*}
F^2_0:= G^2(\lambda^0 F^0_0G^0)^\frac{1}{3}(\lambda^3F^3_0G^3)^\frac{2}{3},& \;F^2_1:= G^2(\lambda^1 F^0_1G^1)^\frac{1}{3}(\lambda^4F^3_1G^4)^\frac{2}{3}, \\
F^4_0:= G^4(\lambda^3 F^3_0G^3)^\frac{2}{3}(\lambda^6F^6_0G^6)^\frac{1}{3},& \;F^4_1:= G^4(\lambda^4 F^3_1G^4)^\frac{2}{3}(\lambda^7F^6_1G^7)^\frac{1}{3}, \\
F^5_0:= G^5(\lambda^3 F^3_0G^3)^\frac{1}{3}(\lambda^6F^6_0G^6)^\frac{2}{3},& \;F^5_1:= G^5(\lambda^4 F^3_1G^4)^\frac{1}{3}(\lambda^7F^6_1G^7)^\frac{2}{3}. 
\end{align*}
\end{proof}

\begin{proof}[Proof of Lemma \ref{thm-flow-smooth}]
Since the $L^2$ energy of the flow decays, it suffices to prove \eqref{tleqL}.
For any $t\in (0, T]$, there exists a unique $n\in \mathbb{Z}$ such that $t\in (2^n, 2^{n+1}]$. Then, thanks to Lemma \ref{thm-flow-con}, we can find some $t'\in (2^{n-1}, 2^{n}]$ satisfies 
\begin{equation}
\|S(t')f_0\|_{H^{k+1}_{(0)}}\leq F_1^k 2^{-(n-1)/2}\|f_0\|_{H^k_{(0)}}.
\end{equation} 
Otherwise, we have 
\begin{align*}
 \int_{0}^{T}  \|S(t')f_0)\|_{H^{k+1}_{(0)}}^2  \,dt&\geq \int_{2^{n-1}}^{2^{n}} \|S(t')f_0)\|_{H^{k+1}_0}^2 \,dt, \\
&> \int_{2^{n-1}}^{2^{n}} (F_1^k)^2 2^{-(n-1)} \|f_0\|_{H^k_{(0)}}^2\,dt, \\
&= \left(F_1^k \|f_0\|_{H^k_{(0)}}\right)^2,
\end{align*}
which is in contradiction with \eqref{F1}. Thanks to inequality \eqref{F0}, we get
\begin{align}\label{01smooth}
\|S(t)f_0\|_{H^{k+1}_{(0)}}&= \|S(t-t')\big(S(t')f_0\big)\|_{H^{k+1}_{(0)}},  \notag\\
&\leq F_0^{k+1}\|S(t')f_0\|_{H^{k+1}_{(0)}},\notag \\
&\leq 2^{-(n-1)/2} F_1^k F_0^{k+1}\|f_0\|_{H^k_{(0)}}, \notag\\
&\leq 2 t^{-1/2}F_1^k F_0^{k+1}\|f_0\|_{H^k_{(0)}}, \;\forall \;t\in (0, T],\; T\leq L.
\end{align}
By applying \eqref{01smooth} with $k=0, 1,..., k$ respectively, we are able to get
\begin{align*}
\|S(t)f_0\|_{H^{n}_{(0)}}&=\Big{\|}\left(S\Big(\frac{t}{n}\Big)\right)^n f_0\Big{\|}_{H^{n}_{(0)}}, \\
&\leq 2 \left(\frac{t}{n}\right)^{-1/2}F^{n-1}_1F^n_0\Big{\|}\left(S\Big(\frac{t}{n}\Big)\right)^{n-1} f_0\Big{\|}_{H^{n-1}_{(0)}}, \\
&\leq 4 \left(\frac{t}{n}\right)^{-1}F^{n-1}_1F^n_0F^{n-2}_1F^{n-1}_0\Big{\|}\left(S\Big(\frac{t}{n}\Big)\right)^{n-2} f_0\Big{\|}_{H^{n-2}_{(0)}},\\
&\leq \frac{2^n n^{n/2}}{t^{n/2}}\left(\prod_{i=0}^{n-1}F^{i}_1F^{i+1}_0\right) \|f_0\|_{L^2},  \;\forall \;t\in (0, T],\; T\leq L.
\end{align*} 
Hence, we conclude the proof of Lemma \ref{thm-flow-smooth} by selecting 
\begin{equation}
F^k_s:=  2^k k^{k/2}\left(\prod_{i=0}^{k-1}F^{i}_1F^{i+1}_0\right), \;k\in \{1, 2, 3, 4, 5, 6\}.
\end{equation}
\end{proof}

\bibliographystyle{plain} 
\bibliography{KdVEffective}

\begin{thebibliography}{10}

\bibitem{Adams-book}
Robert~A. Adams.
\newblock {\em Sobolev spaces}.
\newblock Academic Press [A subsidiary of Harcourt Brace Jovanovich,
  Publishers], New York-London, 1975.
\newblock Pure and Applied Mathematics, Vol. 65.

\bibitem{Bardo-Lebeau-Rauch}
Claude Bardos, Gilles Lebeau, and Jeffrey Rauch.
\newblock Sharp sufficient conditions for the observation, control, and
  stabilization of waves from the boundary.
\newblock {\em SIAM J. Control Optim.}, 30(5):1024--1065, 1992.

\bibitem{bsz}
Jerry~L. Bona, Shu~Ming Sun, and Bing-Yu Zhang.
\newblock A nonhomogeneous boundary-value problem for the {K}orteweg-de {V}ries
  equation posed on a finite domain.
\newblock {\em Comm. Partial Differential Equations}, 28(7-8):1391--1436, 2003.

\bibitem{Brezis-book}
Haim Brezis.
\newblock {\em Functional analysis, {S}obolev spaces and partial differential
  equations}.
\newblock Universitext. Springer, New York, 2011.

\bibitem{cerpa07}
Eduardo Cerpa.
\newblock Exact controllability of a nonlinear {K}orteweg-de {V}ries equation
  on a critical spatial domain.
\newblock {\em SIAM J. Control Optim.}, 46(3):877--899 (electronic), 2007.

\bibitem{cerpatu}
Eduardo Cerpa.
\newblock Control of a {K}orteweg-de {V}ries equation: a tutorial.
\newblock {\em Math. Control Relat. Fields}, 4(1):45--99, 2014.

\bibitem{cerpa09}
Eduardo Cerpa and Emmanuelle Cr{\'e}peau.
\newblock Boundary controllability for the nonlinear {K}orteweg-de {V}ries
  equation on any critical domain.
\newblock {\em Ann. Inst. H. Poincar\'e Anal. Non Lin\'eaire}, 26(2):457--475,
  2009.

\bibitem{sobolev-quantitative}
S.~N. Chandler-Wilde, D.~P. Hewett, and A.~Moiola.
\newblock Interpolation of {H}ilbert and {S}obolev spaces: quantitative
  estimates and counterexamples.
\newblock {\em Mathematika}, 61(2):414--443, 2015.

\bibitem{Lebeau-sharp-2016}
Felipe~W. Chaves-Silva and Gilles Lebeau.
\newblock Spectral inequality and optimal cost of controllability for the
  {S}tokes system.
\newblock {\em ESAIM Control Optim. Calc. Var.}, 22(4):1137--1162, 2016.

\bibitem{coron15}
Jixun Chu, Jean-Michel Coron, and Peipei Shang.
\newblock Asymptotic stability of a nonlinear {K}orteweg--de {V}ries equation
  with critical lengths.
\newblock {\em J. Differential Equations}, 259(8):4045--4085, 2015.

\bibitem{coron}
Jean-Michel Coron.
\newblock {\em Control and nonlinearity}, volume 136 of {\em Mathematical
  Surveys and Monographs}.
\newblock American Mathematical Society, Providence, RI, 2007.

\bibitem{coron04}
Jean-Michel Coron and Emmanuelle Cr{\'e}peau.
\newblock Exact boundary controllability of a nonlinear {K}d{V} equation with
  critical lengths.
\newblock {\em J. Eur. Math. Soc. (JEMS)}, 6(3):367--398, 2004.

\bibitem{coronluqi}
Jean-Michel Coron and Qi~L{\"u}.
\newblock Local rapid stabilization for a {K}orteweg-de {V}ries equation with a
  {N}eumann boundary control on the right.
\newblock {\em J. Math. Pures Appl. (9)}, 102(6):1080--1120, 2014.

\bibitem{2017-Coron-Nguyen-ARMA}
Jean-Michel Coron and Hoai-Minh Nguyen.
\newblock Null controllability and finite time stabilization for the heat
  equations with variable coefficients in space in one dimension via
  backstepping approach.
\newblock {\em Arch. Ration. Mech. Anal.}, 225(3):993--1023, 2017.

\bibitem{coron-1991-adding-integral}
Jean-Michel Coron and Laurent Praly.
\newblock Adding an integrator for the stabilization problem.
\newblock {\em Systems Control Lett.}, 17(2):89--104, 1991.

\bibitem{coron-rivas-xiang-kdv-16}
Jean-Michel Coron, Ivonne Rivas, and Shengquan Xiang.
\newblock Local exponential stabilization for a class of {K}orteweg--de {V}ries
  equations by means of time-varying feedback laws.
\newblock {\em Anal. PDE}, 10(5):1089--1122, 2017.

\bibitem{Coron-Xiang-2018}
Jean-Michel Coron and Shengquan Xiang.
\newblock Small-time global stabilization of the viscous {B}urgers equation
  with three scalar controls.
\newblock {\em Preprint, hal-01723188}, 2018.

\bibitem{Zuazua-cost-2011}
Sylvain Ervedoza and Enrique Zuazua.
\newblock Sharp observability estimates for heat equations.
\newblock {\em Arch. Ration. Mech. Anal.}, 202(3):975--1017, 2011.

\bibitem{Fursikov-Imanuvilov-book-1997}
Andrei~V. Fursikov and Oleg~Yu. Imanuvilov.
\newblock {\em Controllability of evolution equations}, volume~34 of {\em
  Lecture Notes Series}.
\newblock Seoul National University, Research Institute of Mathematics, Global
  Analysis Research Center, Seoul, 1996.

\bibitem{2008-Krstic-Smyshlyaev-book}
Miroslav Krstic and Andrey Smyshlyaev.
\newblock {\em Boundary control of {PDE}s}, volume~16 of {\em Advances in
  Design and Control}.
\newblock Society for Industrial and Applied Mathematics (SIAM), Philadelphia,
  PA, 2008.
\newblock A course on backstepping designs.

\bibitem{Lebeau-Robbiano-CPDE}
Gilles Lebeau and Luc Robbiano.
\newblock Contr\^ole exact de l'\'equation de la chaleur.
\newblock {\em Comm. Partial Differential Equations}, 20(1-2):335--356, 1995.

\bibitem{Lions-Hilbert}
Jacques-Louis Lions.
\newblock {\em Contr\^{o}labilit\'{e} exacte, perturbations et stabilisation de
  syst\`emes distribu\'{e}s. {T}ome 2}, volume~9 of {\em Recherches en
  Math\'{e}matiques Appliqu\'{e}es [Research in Applied Mathematics]}.
\newblock Masson, Paris, 1988.
\newblock Perturbations. [Perturbations].

\bibitem{Lissy-2014}
Pierre Lissy.
\newblock On the cost of fast controls for some families of dispersive or
  parabolic equations in one space dimension.
\newblock {\em SIAM J. Control Optim.}, 52(4):2651--2676, 2014.

\bibitem{zuazua02}
Gustavo~Alberto Perla~Menzala, C.~F. Vasconcellos, and Enrique Zuazua.
\newblock Stabilization of the {K}orteweg-de {V}ries {E}quation with localized
  damping.
\newblock {\em Q. Appl. Math.}, LX(1):111--129, 2002.

\bibitem{rosier97}
Lionel Rosier.
\newblock Exact boundary controllability for the {K}orteweg-de {V}ries equation
  on a bounded domain.
\newblock {\em ESAIM Control Optim. Calc. Var.}, 2:33--55 (electronic), 1997.

\bibitem{Stein}
Elias~M. Stein.
\newblock {\em Singular integrals and differentiability properties of
  functions}.
\newblock Princeton Mathematical Series, No. 30. Princeton University Press,
  Princeton, N.J., 1970.

\bibitem{2009-Tucsnak-Weiss-book}
Marius Tucsnak and George Weiss.
\newblock {\em Observation and control for operator semigroups}.
\newblock Birkh\"auser Advanced Texts: Basler Lehrb\"ucher. [Birkh\"auser
  Advanced Texts: Basel Textbooks]. Birkh\"auser Verlag, Basel, 2009.

\bibitem{2017-Xiang-SCL}
Shengquan Xiang.
\newblock Small-time local stabilization for a {K}orteweg-de {V}ries equation.
\newblock {\em Systems \& Control Letters}, 111:64 -- 69, 2018.

\bibitem{2019-xiang-SICON}
Shengquan Xiang.
\newblock Null controllability of a linearized {K}orteweg-de {V}ries equation
  by backstepping approach.
\newblock {\em SIAM J. Control Optim.}, 57(2):1493--1515, 2019.

\bibitem{ZhangRapidStab}
Christophe Zhang.
\newblock {Internal rapid stabilization of a 1-D linear transport equation with
  a scalar feedback}.
\newblock Preprint, October 2018.

\end{thebibliography}

\end{document}